\documentclass{amsart}

\usepackage{amssymb,latexsym} %
\theoremstyle{plain}
\newtheorem{theorem}{Theorem}
\newtheorem{lemma}{Lemma}
\newtheorem{corollary}{Corollary}
\newtheorem{proposition}{Proposition}

\theoremstyle{definition}
\newtheorem{definition}{Definition}

\theoremstyle{remark}
\newtheorem{remark}{Remark}

\numberwithin{equation}{section}
\begin{document}
\title[Yang Mills Higgs]{Removeability of a  Codimension Four Singular Set for  Solutions of a Yang Mills Higgs Equation with Small Energy}
\author{Penny Smith}
\curraddr{
Department of Mathematics Lehigh University 14 E. Packer Ave, Bethehem Pa. 18015 }
\email{
ps02@lehigh.edu, penny314@gmail.com }
\author{Karen Uhlenbeck}
\curraddr{
School of Mathematics Institute for Advanced Study, 1 Einstein Drive
Princeton, New Jersey
08540 USA}
\email{
uhlen@math.utexas.edu, karen.uhlen42@gmail.com}
\date{
June 20, 2018
}
\today
\subjclass[2010]{Primary 05C38, 15A15; Secondary 05A15, 15A18}
\keywords{Keyword one, keyword two, keyword three}
\begin{abstract}
The Yang-Mills-Higgs equations, while they arose in theoretical physics, have become a central object of study in geometric analysis.  The moduli spaces of solutions  are used to study problems in algebraic geometry as well as topology.  To understand these moduli spaces, it is natural to ask what sort of objects are limits in a weak sense of sequences of solutions.  In many examples, it can be shown that subsequences of solutions converge to a solution off a set of Hausdorff codimension 4.  In four dimensions, these point singularities are removable,  The question arises of when this singular set of dimension n-4 is removable in higher dimensions.
In a classic paper, Tao and Tian  [10] prove an important step in the removability theorem.  What they show is that, in general, independent of the equation, if the case the curvature of a connection defined on a set of Hausdorff codimension 4 is sufficiently small in a certain Morrey space, that there exists a smaller ball and a gauge in which the connection is $d + A$, where $A$ is coclosed and bounded by norms on the curvature. One can now treat the Yang-Mills-Higgs equations as a standard elliptic system and obtain further regularity.  This choice of Morrey space is natural, because monotonicity of solutions provides estimates exactly in this space.  However, since the solution  is only defined off a set of Hausdorff codimension 4, it is not known that the monotonicity theorem is true for limiting singular solutions. Only by adding the assumption of stationary with respect to diffeomorphisms in the entire domain can one apply the theorem to limiting singular connections.
Unfortunately, we cannot fill in this gap.  We provide a new proof of the Tao-Tian result.  Our proof  uses the equation and a maximum principle for a differential inequality derived from the equation via averaging.   In this manner, we can avoid the geometric construction of a good gauge in a weak setting via averaging exponential gauges, as we first obtain further estimates on the curvature from the inequality.  Our main result ( Corollary \ref{C:scaledandtranslatedgaugeequivalentregularitytheorem} ) is that, if the curvature is sufficiently small in a Morrey space, and the solution is defined and smooth  off a set of Hausdorff codimension 4 with the covariant derivative of the Higgs field in the same Morrey space, then the solution is smoothly gauge equivalent to one which extends smoothly over the singular set. In the case that the solution has small energy, and is a critical point of the integral, or even just critical with respect to diffeomorphisms in the entire domain, not just off the singular set, the smallness of the curvature in the Morrey space in a smaller domain can be verified and leads to Theorem 11 as a corollary of Corollary \ref{C:scaledandtranslatedgaugeequivalentregularitytheorem}.  There may be other ways to obtain the bounds on the curvature in the Morrey space; hence we state the main result in terms of small curvature  in the Morrey space. 
The results apply to most of the coupled Yang-Mills-Higgs equations, and we do not go into the details in this paper.  We do assume that the Higgs field is bounded, and explain 
why this follows from $L^2$ bounds on it.   Many interesting equations, such as the Kapustin-Witten equations and the equations for a complex flat connection do not come with such a bound on sequences.  Moreover, the singular set of a renormalized limit is Hausdorff codimension 2, and we sadly have nothing to say about these singularities..

We develop a new method for proving regularity for small energy stationary solutions of coupled gauge field equations. Our results duplicate those of [10] for the pure Yang Mills equations, but our proof is simpler, and obtains bounded curvature without the use of Coulomb gauges. It relies instead on the Weitzenblock formulae, and an improved Kato inequality. Our results also extend and simplify those of [3].
\end{abstract}
\maketitle
\section{Introduction
}\label{S:ScalingandmaxPHI}
 We consider stationary solutions of the Euler Lagrange equations for the integral
 \begin{equation}\label{E:eulerlagrange}
 \mathcal{A}(D_{A},\Phi)=
 \int_{\Omega} \left( \vert F_{A} \vert^{2} +
 \vert D_{A} \Phi \vert^{2} +
 Q(\Phi) \right) \, (dx)^{n},
\end{equation}
for $\Omega \subset \mathbb{R}^{n}$ $n \geq 4$, and $D_{A}$ is a unitary connection
on the product bundle $V \times 
\Omega$, $\Phi$ is a section of the product bundle $V \times \Omega$ , and where $Q(\bullet)$ is a real valued equivariant function from sections of the product bundle $V \times \Omega$.
The integrand on the right hand side of \eqref{E:eulerlagrange} is gauge invariant, but since the result is local, we do not need any of the topological considerations in its formulation. 
A typical term in $Q(\Phi )$ can be 
$\vert ( \vert \Phi \vert^{2} - \lambda^{2} )\vert^{2} $ if $V = \mathfrak{su}(N)$,
but we impose only three conditions on $Q$, namely

\begin{subequations}\label{E:conditionsonQ}
\begin{equation}\label{E:firstconditiononQ}
Q(\Phi) \geq 0 
\end{equation}
\begin{equation}\label{E:secondconditiononQ}
(\Phi,  Q_{\Phi}(\Phi ) )
\geq -K_{1}^{2}. 
\end{equation}

Here $Q_{\Phi}(\Phi )$ is the directional derivative of the function $Q$, in the
direction $\phi$, and, hence, $Q_{\Phi}(\Phi )$ is a section of $V^{*}$. By the usual duality argument, we can identify it with a section of $V$.
\begin{equation}\label{E:groupinvarianceofQ}
Q(\Phi) \, \text{is equivariant with respect to the group action.}
\end{equation}
\end{subequations}

First, we state the Euler--Equations for the energy functional \eqref{E:eulerlagrange}
\begin{subequations}\label{E:eulerlagrangeeqs}
\begin{equation}\label{E:firsteulerlagrangeeq}
2(D_{A} )^{*} F_{A} = [\Phi, D_{A} \Phi] 
\end{equation}
\begin{equation}\label{E:secondeulerlagrangeeq}
2(D_{A})^{*} D_{A} \Phi =  Q_{\Phi}(\Phi	).
\end{equation}
\end{subequations}

As in the previous work on this subject, monotonicity theorems and change of scale are key ingredients.
We generally only change scale by blow-up, from $\vert x - x_{0} \vert  \leq r$ to $ \vert y \vert = \left(\frac{\vert x- x_{0} \vert} {r} \right) R$, for $0 <r <R$.
Since, we are in a local trivialization, $D_{A} = d +A$ is a local covariant exterior
derivative with fixed scaling given by
\begin{equation}\label{E:scalingforA}
A dx = \tilde{A} dy = (\tilde{A})( \frac{R}{r}) dx,
\end{equation}
which implies
$\tilde{A} = (\frac{r}{R}) A$. Consistency of the first two terms in the integrand on the right hand side of \eqref{E:eulerlagrange} demands that
$\Phi(x) =\tilde{\Phi} (y) (\frac{r}{R})$
,that is that $\Phi$ scales like a 1-form.
This requires $Q(\Phi) $ to scale like a 4-form.
That is
\begin{equation}\label{E:Qscaling}
\tilde{Q_{R}}(\tilde{\Phi}) = \left(\frac{R}{r}\right)^{4} (Q_{r})(\Phi)= 
\left(\frac{R}{r}\right)^{4} Q(\left(\frac{r}{R}\right) \tilde{\Phi}).	\end{equation}
This looks formidable, however, we are saved by \eqref{E:conditionsonQ} and Theorem
\eqref{T:secondpointwiseboundonphi}, which provides bounds on terms depending on $\vert \Phi \vert$, and hence bounds on terms depending on $Q$.

Singular solutions of \eqref{E:eulerlagrangeeqs} may arise as limits of smooth solutions.
The following theorem indicates that, if there is a uniform bound in $L^{2}$
on the sequence of Higgs Fields $ \Phi_{j}$, there will also be a bound on the maximum of $ \vert \Phi_{j} \vert$, and hence on the maximum of $ \vert \Phi \vert$ in the singular limit.

\begin{theorem}\label{T:secondpointwiseboundonphi}
 Let $ ( D_{A}, \Phi )$ be a smooth solution of the field equations \eqref{E:eulerlagrange} in the domain $\Omega \subset \mathcal{R}^{n}$.
 Let, $  \Phi \bullet Q_{\Phi}(\Phi) \Phi  \geq  -(K_{1}) $, where $ \bullet$ denotes inner product, 
 and let $ \int_{\Omega} \vert \Phi \vert^{2} \, (dx)^{n} \leq (K_{2})^{2}  $. Then $\vert \Phi \vert$ is  bounded in the interior of $ \Omega$ and for all $ x_{0} \in int(\Omega)$, $ 0< d < dist(x_{0}, \partial \Omega)$
\begin{equation}\label{E:basicmorreyboundonphi}
\vert \Phi(x_{0}) \vert^{2}  \leq 
((d)^{-n})C_{1}(n) 
) 
\int_{\vert x-x_{0}\vert \leq d}
  \vert \Phi \vert^{2}\, (dx)^{n} +
  C_{2}(n) K^{2}_{1} d^{2}.
\end{equation}
\end{theorem}
\begin{proof}
We will make use of the following identity:

\begin{equation}\label{E:laplacianidentity}
(\frac{1}{2}) \triangle \vert \Phi \vert^{2} = (\Phi) \bullet (D^{*}_{A}D_{A}
\Phi) + \vert D_{A} \Phi\vert^{2}.
\end{equation}
The equation \eqref{E:laplacianidentity} follows from the facts that $\Phi$ is locally a bundle section and the compatibilty of the connection with the inner product on sections.

Here $\triangle$ is the co-ordinate laplacian on the base.
We use the the field equation \eqref{E:secondeulerlagrangeeq} to replace the first term on the right hand side of \eqref{E:laplacianidentity}, and obtain
\begin{equation}\label{E:modifiedlaplacianestimatefornormphisquared}
 \triangle \vert \Phi \vert^{2} = (\Phi) \bullet ( Q_{\Phi}(\Phi)) +  2\vert D_{A} \Phi\vert^{2}.
\end{equation}
 Using the lower bound $ \Phi 
 \bullet 
 Q_{\Phi}(\Phi) 
  \geq  -(K_{1})^{2} $ in 
\eqref{E:modifiedlaplacianestimatefornormphisquared}, we obtain:
\begin{equation}\label{E:almostsubharmonicboundforphisquared}
 \triangle \vert \Phi \vert^{2} \geq -K_{1}^{2}.
\end{equation}
Thus, the function
$\vert \Phi \vert^{2} +
 (x-x_{0})^{2}
 K_{1}
 ^{2} $ is non-negative smooth subharmonic function. Now apply the sub-mean value property for subharmonic functions, the triangle inequality for integrals, the triangle inequality for sums, and take the sup of $(x-x_{0})^{2}K_{1}^{2}$ inside the integral over $B_{d}(x_{0}) \subset \subset \Omega$.
	\end{proof}

We can use these two estimates in several ways, depending on the actual structure of
$ Q(\Phi)$.
The terms that we have in mind are terms like
$ ( \vert \Phi \vert^{2} -a^{2} )^{p}$, or $\vert [\Phi, \Phi ] \vert^{p} $, for $ p \geq 1$, where the simple cases occur for $p =2$.

We just gave a description of the functional and motivated the bound that we assume on the maximum of $\Phi$ in subsequent chapters. In chapter 2 we outline a proof of the monotonicity theorem. This motivates the assumption we place on the Morrey Space norm of the curvature $F_{A}$ in our final theorem.
We emphasis that for its validity, it requires the critical point $( D_{A}, \Phi)$ to be stationary with respect to diffeomorphisms that are the identity near the boundary of the domain.
Chapter 3 contains the necessary improvement to the usual Kato inequality
$ \vert d \vert \upsilon \vert \vert \leq \vert \nabla \upsilon \vert$
where $\upsilon = (F_{A}, D_{A} \Phi)$, as well as a reminder of the 
Weitzenbock formulae for the same quantity. In this section, as in chapter 1,
our computations hold only on the set in which the connection is smooth.
Estimates for smooth solutions are contained in chapter 4.
Chapter 4 is not essential but the estimates for smooth solutions with curvature small in a Morrey Space are easier to obtain and provide a 
warm-up for the singular case.

The core of the paper is in Chapter 5. Here we use the properties of Morrey Spaces to get regularity for singular solutions of a differential inequality (on functions $f$)with a coefficent small in a borderline Morrey space. The theorem that we prove is Theorem 9:

\textbf{Theorem 9} :\\
Let $u \geq 0$ and $f>0$ be smooth functions on $\Omega_{4} - \mathcal{S}$,
where $\mathcal{S}$ is a closed set of finite $n-4$ Hausdorff dimension.
If
\begin{equation}\label{E:singularsubinequality}
-\triangle f + \alpha \left( \frac{\vert df \vert }{f} \right) -uf \leq Qf,
\end{equation}
then there exist constants $\eta_{k} > 0$, and $\kappa_{k} > 0$, such that
if $\Vert u \Vert_{X^{2}(\Omega_{4}) } < \eta_{k}$, then $f \in X^{k}(\Omega_{1} )$. Moreover

\begin{equation}\label{E:morreyboundonfsingularcase}
\Vert f \Vert_{X^{k}(\Omega_{1} ) } \leq \kappa_{k} 
\Vert f \Vert_{X^{2}(\Omega_{4}) }.
\end{equation}

Chapter 6 is a straightforward application of Theorem 9 and the improved Kato inequality of chapter 3.

The idea of using differential inequalities to get
estimates of the norms of curvature and covariant derivative of the
Higgs field, and to  thus to remove high codimension singular sets  is
due to Tom Otway [5].

We have Theorem 10:

\textbf{Theorem 10} :\\
Let $( D_{A},  \Phi )$ be solutions to a Yang-Mills-Higgs system in 
$\Omega_{4} -\mathcal{S} $, where $\mathcal{S}$ is a closed set of finite $n-4$
dimensional Hausdorff measure. Let $ \upsilon = ( F_{A}, D_{A} \Phi )$. Assume that $\upsilon \in X^{2}(\Omega_{4} )$, and that $\Phi \in L^{\infty}(\Omega_{4} )$. 
\begin{equation}\label{E:definitionofQsub1}
Q_{1}  = \underset{[-4,4]^{n}}{sup}
( 2 \vert \Phi \vert + 
 \vert Q_{\Phi,\Phi}(\Phi) ) \vert
\end{equation}
\begin{equation}\label{E:definitionofQsub2}
Q_{2}^{2} = \underset{[-4,4]^{n}}{sup} 
\left(\frac{\vert  Q_{\Phi}(\Phi)\vert^{2}}{Q_{1}}
\right).
\end{equation}
If $F_{A} \in X^{2}(\Omega_{4})$ is sufficently small, then $ \vert \upsilon \vert  \in L^{\infty}(\Omega_{1})$.
We also have the explicit bound 
\begin{equation}\label{E:L2nboundonupsilon}
\Vert \upsilon \Vert_{L^{2n}(\Omega_{1})} \leq
C(Q_{1}) ( \Vert \upsilon \Vert_{X^{2}(\Omega_{4})} + Q^{2}_{2} )
\end{equation}

We state the main regularity theorem as a corollary of Theorem 10.
It applies the Coulomb gauge construction of appendix C.
We have there:
 \\
\textbf{Corollary 4} :\\
Assume $( D_{\tilde{A}}, \Phi )$ satisfies a Yang-Mills-Higgs system in $\Omega - \mathcal{S}$, where $\mathcal{S}$ is a closed set of finite $n-4$ dimensional Hausdorff measure.
Let $\upsilon = ( F_{\tilde{A}}, D_{\tilde{A}} \Phi )$
 
\begin{equation}\label{E:defintionofmodifiedQ1}
Q_{1} = \underset{\Omega}{sup} \left( \vert  Q(\Phi) \vert^{2} +\ \vert  Q_{\Phi,\Phi}(\Phi) \vert \right).
\end{equation}

\begin{equation}\label{E:definitionofmodifiedQ2}
Q_{2}^{2} = \frac {\left( \underset{\Omega}{sup} \vert Q_{\Phi}(\Phi) \vert^{2}
	\right)}{ Q_{1} }.
\end{equation}
Suppose $\upsilon \in X^{2}(\Omega)$, and $Q_{1} \delta^{2}$ as well as
$Q_{2} \delta^{2}$ (scales like the two form $\upsilon$) are bounded by a fixed constant. In addition suppose that $ F_{\tilde{A}} \in X^{2}(\Omega)$ has small enough $X^{2}(\Omega)$ norm (independent of the other constants).
If $\Omega_{y,\delta} \subset \Omega$, then $\delta^{2} \upsilon \in L^{\infty}(\Omega_{y, \delta})$ is bounded above by a constant, and $
(D_{\tilde{A}}, \Phi )$
are smoothly gauge equivalent on 
$\Omega_{y,\delta}- \mathcal{S} $
 to a smooth exterior covariant differential (corresponding to a smooth connection), and a smooth Higgs Field on $\Omega_{y,\delta}$, both of which extend smoothly across the singular set $ \mathcal{S} \cap \Omega_{y,\delta} $.
\\

The Appendices are important. In Appendix A we outline the necessary background
on Morrey Spaces. We refer the reader to the book by Adams [1], and give outlines of applications that are not in this reference, such as the invertibility of the Laplacian on cubes with Dirichlet boundary value.
We introduce notation such as the space $X^{k} = M^{[\frac{n}{k},  \frac{4}{k} ]}$ for the Morrey space which scales like $\frac{n}{k}$ and has integral power $\frac{4}{k}$. This simplifies our exposition.
Appendix B contains the maximum principle that we use in chapter 5, which seemed to fit better in an appendix than in the body of the paper.
Appendix C is the proof of the existence of a Coulomb gauge for singular connections. This is not in the standard literature because we allow a singular set of Hausdorff codimension three (We only need Hausdorff codimension four for our application). Tao and Tian [10] prove this with much weaker conditions but their proof is much more elaborate.
Since the point of this paper is to simplify their proof, we include it.

\section{Monotonicity Formulae}\label{S:monotonicityformulae}

In this section we use the condition that $ \mathcal{A}(D_{A}, \Phi)$ 
is stationary in the sense of Price [4], and Zhang [9]  at $(D_{A},\Phi)$ with respect to smooth diffeomorphisms, that are the identity near the domain boundary, to derive 
a monotonicity formula. We give an outline, since the method is the same as used by many authors, cf. [4], and Lemma 2.1 of [11]. 

\begin{theorem}\label{T:monotonicitytheorem}
If $(D_{A},\Phi )$ is
 a stationary point of 
$ \mathcal{A}(D_{A}, \Phi)$ with respect to smooth diffeomorphisms of $\Omega$ which equal the identity in a neighborhood of $\partial \Omega $, then for all smooth vector fields $ \boldsymbol {\nu}$ with compact support in $ \Omega$, we have

\begin{equation}\label{E:nuidentity}
\begin{split}
\int_{\Omega} 
[
\partial_{k} \boldsymbol{ \nu}^{k} \vert F_{A} \vert^{2} + \vert D_{A} \Phi \vert^{2} + Q(\Phi)- \partial_{j} \boldsymbol{ \nu}^{k} F_{ki} F_{ji} \\
-2 \partial_{j} \boldsymbol {\nu}^{k} (D_{A,j}\Phi) (D_{A,k}\Phi) \, ] (dx)^{n}=0
\end{split}.
\end{equation}
\end{theorem}
\begin{proof}
This formula can be straightforwardly derived for smooth solutions $F_{A}, \Phi $ via
Noether's formula. But, by direct calculation (compare formula (2.4) page 265 of [11]), it is also true for $\vert F_{A} \vert^{2},
\vert 
\Phi \vert^{2}
, Q(\Phi ) 
\in L^{1}(\Omega)  $
Thus, if the singular points of $\vert F_{A} \vert^{2},
\vert D_{A} 
\Phi \vert^{2}
, Q(\Phi ) 
\in L^{1}(\Omega)  $, occur on a set of measure zero, they do not affect formula
\eqref{E:nuidentity}
\end{proof}

\begin{corollary}\label{C:firstvariationalidentity}
If \eqref{E:nuidentity} holds, and if $\mathcal{A}(D_{A}, \Phi)$ is finite, then, for $ \{ \vert x-x_{0} \vert \leq R \} \subset \Omega  $, we have
\begin{equation}\label{E:firstvariationalidentity}
\begin{split}
\int_{\vert x-x_{0} \vert \leq R } 
[ (n-4) \vert F_{A} \vert^{2} +(n-2) \vert D_{A} \Phi \vert^{2} + nQ(\Phi) ]
\, (dx)^{n} - \\
R \int_{\vert x-x_{0} \vert = R } [\vert F_{A} \vert^{2} + \vert D_{A} \Phi \vert^{2} +Q(\Phi) -4 \vert F_{A,r} \vert^{2} -2 \vert D_{A,r} \Phi \vert^{2} ]
\, (dx)^{n-1} =0
\end{split}.
\end{equation}
Here, $ F_{A,r} = i_{\frac{\partial}{\partial r}} (F_{A}) $ is the radial part of $F_{A}$,
and $D_{A,r} =  D_{A} ( \frac{\partial}{\partial r} ) $ is the radial part of
$D_{A}$.
\end{corollary}
\begin{proof}
With no loss of generality we assume that $x_{0} =0$, since our calculations are translation invariant.
Let $\boldsymbol{ \nu } = \boldsymbol{\nu}_{\epsilon} $ be a smooth vector field, dependent on a small positive parameter $\epsilon$, where $\boldsymbol{\nu}$ is defined by
\begin{equation}\label{E:definitionofnu}
\begin{cases}
\boldsymbol{\nu} = r \frac{\partial}{\partial r}  & \text{if} \quad \vert x \vert  \leq R- \epsilon \\
\boldsymbol{\nu} = 0 & \text{if} \quad \vert x \vert \geq R \\
\boldsymbol{\nu} = \eta (\frac{R- \vert x \vert}{\epsilon} ) r \frac{\partial}{\partial r} & \text{otherwise,} 
\end{cases}
\end{equation}
Here $ \eta $
is a smooth function, satisfying 
\begin{equation}\label{E:definitionofphi}
\begin{cases}
\eta(t) = 0 & \quad \text{for} \quad t \leq 0 \\
\eta(t)= 1 & \quad \text{for} \quad t \geq 1
\end{cases}
\end{equation}
with  $\phi^{\prime} \geq 0$.
Note that:
\begin{equation}\label{E:dnuidentity}
\frac{\partial}{\partial x^{k}} \boldsymbol{\nu}^{i} = \eta (\frac{R- \vert x \vert}{\epsilon} ) \delta^{i}_{k} 
- (\frac{1}{\epsilon} ) \eta^{\prime} (\frac{R- \vert x \vert}{\epsilon} ) 
\frac{x^{i} x^{k}}{\vert x \vert}. 
\end{equation}

Using this $\boldsymbol{\nu} $ in \eqref{E:nuidentity}, we obtain
\begin{align}\label{E:firstpremonotonicityidentity}
& \int_{\vert x \vert \leq R}
  \eta (\frac{R- \vert x \vert}{\epsilon} )
[
(n-4) \vert F_{A} \vert^{2} + (n-2) \vert D_{A} \Phi \vert^{2} + nQ(\Phi) 
] \, (dx)^{n} -  \\ \notag
&\int_{R-\epsilon \leq \vert x \vert = \rho \leq R } 
(\frac{1}{\epsilon} )
\eta^{\prime} 
(\frac{R- \vert x \vert}{\epsilon} ) [ \vert F_{A} \vert^{2}+ 
\vert D_{A} \Phi \vert^{2} +Q(\Phi) - 4 \vert F_{A} \vert^{2} - 2 \vert D_{A,r} 
\Phi \vert^{2}  ] 
\rho  \, 
(da)_{n-1} 
d \rho.
\end{align}
 where $ (da)_{n-1}= \rho^{n-1} (d \Theta)$ is the area element on $ \vert x \vert = \rho$.

Let $t = R- \rho$, and note that $ lim_{\epsilon \downarrow 0 } (\frac{1}{\epsilon} ) ( \eta^{\prime}( \frac{t}{\epsilon}) )\rightarrow \delta $, in the sense of distributions, where $ \delta$ is the delta distribution.
Thus, letting $\epsilon \downarrow 0 $ gives our result.
\end{proof}
From Corollary \ref{C:firstvariationalidentity}, by ignoring the radial parts
of $F_{A}$, and $D_{A}$ which appear with a negative sign, and by replacing both
$n$ and $n-2$ by $n-4$, we derive a differential inequality on
\[
\mathcal{E}(R) = \int_{\vert x \vert \leq R}
\vert F_{A} \vert^{2} + \vert D_{A} \Phi \vert^{2} + Q(\Phi) \, (dx)^{n}.
\] 
\begin{equation}\label{E:differentialinequalityforenergy}
(n-4) \mathcal{E}(R) \leq  \mathcal{E}^{\prime}(R).
\end{equation}

Integrating \eqref{E:differentialinequalityforenergy} gives the monotonicity formula
\begin{equation}\label{E:firstmonotonicityformula}
\mathcal{E}(r) \leq (\frac{r}{R})^{n-4} \mathcal{E}(R).
\end{equation}
\begin{theorem}\label{T:monotonicitytheorem}
If $  (D_{A}, \Phi)$ is a stationary point of the functional
$ \mathcal{A} (D_{A}, \Phi)$, with respect to smooth diffeomorphisms of its domain, then, if $ \{\vert x-x_{0} \vert \leq R \} \subset \Omega$  for
$r \leq R$,  we have
\begin{equation}\label{E:secondmonotonicityformula}
	\int_{ \vert x- x_{0} \vert \leq r}
	\vert F_{A} \vert^{2} + \vert D_{A} \Phi \vert^{2} + Q(\Phi) \, (dx)^{n}
	\leq
	\left(\frac{r}{R}\right)^{n-4}
	\int_{ \vert x- x_{0} \vert \leq R}
	\vert F_{A} \vert^{2} + \vert D_{A} \Phi \vert^{2} + Q(\Phi) \, (dx)^{n}
	\end{equation}.	
	\end{theorem}

\begin{remark}
In the case where we have a Riemannian metric, instead of a Euclidean metric,
this formula is easily seen to be valid with an error term.
\end{remark}

In order to prove our final regularity theorem
Theorem \ref{T:finalregularitytheorem}   
(which is a Corollary of 
Lemma \ref{L:X2normboundsonstablesoltions} and 
of Corollary \ref{C:scaledandtranslatedgaugeequivalentregularitytheorem}, [ all in Section 
\ref{S:application to Yang Mills Higgs}],
 we need the scaling estimates below:
 
 We take a small ball $ \vert x - x_{0} \vert \leq r$ to a ball
$\vert y \vert \leq 4$. In this case, as in section \ref{S:ScalingandmaxPHI},
we may assume that we have a bound on $\Phi$, and therefore on $Q(\Phi)$, and its derivatives with respect to $\Phi$.
That is:
\begin{align}\label{E:Qbounds}
& \vert \Phi  \vert \leq	h \\
&Q(\Phi) \leq h_{0} = max_{\vert \Phi \vert \leq h} Q(\Phi) \\
&\vert Q_{\Phi}(\Phi) \vert  
\leq h_{1} = max_{\vert Q_{\Phi}(\Phi) \vert \leq h} \vert Q_{\Phi}({\Phi} \vert.
\end{align}
Under rescaling, from $r \rightarrow R$, denoting the rescaled terms by
$\tilde{\Phi}$, $\tilde{Q}$ ect., 
We obtain
\begin{align}\label{E:scalingtildeidentities}
&\vert  \tilde{\Phi} \vert = \left( \frac{r}{R} \right)\vert \Phi \vert \leq 
\left( \frac{r}{R} \right) h \\ \notag
&\tilde{Q} = \left( \frac{r}{R} \right)^{4} Q \leq  h_{0} \\ \notag
& \vert\tilde{Q}_{\tilde{\Phi}} \vert = \left( \frac{r}{R} \right)^{3} \vert Q_{\Phi} \vert \leq  \left( \frac{r}{R} \right)^{3} h_{1}.
	\end{align}
If $\Phi$ is more regular, then analogous bounds hold for the higher derivatives of $\tilde{Q}$	with respect to $\Phi$. This becomes important for the regularity theory.

 We have a rescaled monotonicity type estimate.
\begin{theorem}\label{T:rescaledmonotonicitytypeestimate}
If
\begin{equation}\label{E:monotonicitytypeestimate}
\int_{\vert x-x_{0} \vert \leq r}
 [\vert F_{ A} \vert^{2} + \vert D_{A}\Phi \vert ^{2} + Q(\Phi)] \, (dx)^{n}
\leq C (r)^{n-4}
\end{equation}
is rescaled to $y = (\frac{R}{r} ) (x- x_{0} )$, we have 
\begin{equation}\label{E:rescaledmonotonicitytypeestimate}
\int_{\vert x-x_{0} \vert \leq R}
[\vert F_{ \tilde{A }} \vert^{2} + \vert D_{\tilde{A} } \tilde{\Phi} \vert ^{2} + \tilde{Q} (\tilde{\Phi} )] \, (dy)^{n}
\leq C (R)^{(n-4)}.
\end{equation}
Moreover, the rescaled variables satisfy \eqref{E:scalingtildeidentities}.
\end{theorem}
\begin{remark}
Under blowup, since we assume a bound on $\Phi$, not only do $F_{ \tilde{A},r} \rightarrow 0$ and $D_{r, \tilde{A} } \Phi \rightarrow 0 $, but
$\tilde{\Phi}$ and $\tilde{Q} \rightarrow 0$. Thus, the theory of blow-ups is the same as for pure Yang-Mills. This is somewhat disappointing.
\end{remark}
There is a direct application of the monotonicity theorem \ref{T:monotonicitytheorem} in the proof of
 Lemma \ref{L:X2normboundsonstablesoltions} and thus in the proof of
Theorem \ref{T:finalregularitytheorem}, which are in
 of Section \ref{S:application to Yang Mills Higgs}.

\section{Improved Kato Inequalities}

Let $\nabla_{A} = \{ \nabla_{i, A} \}$ (where the $i$ refers to local co-ordinates $x^{i}$ on the base ) denote a local covariant derivative in a bundle (where this notation is used to make a clear distinction between the full covariant derivative on the bundle and the exterior covariant differential), and $\nu$ is a $C^{1} $ section, with $\nabla \nu $ continuous, the pointwise inequality
\begin{equation}\label{E:pointwisekatoinequality}
\vert d \vert \nu \vert \vert^{2} \leq \beta \vert \ \nabla \nu \vert^{2} 
\text{, for} \quad \beta =1
\end{equation}
is well known. However, if $\nu$ satisfies some elliptic equations, often the constant
$\beta$ can be improved. This is particularly useful in removing singularities.

Such "Refined Kato inequalities have been used 
before. In particular see (23') page 219 [6] and the estimate of the second fundamental form in [9].

In the following, we will choose a specific orthonormal frame at $x \in \Omega \subset \mathcal{R}^{n}$. We write the inequalities in such a way as to make the extension to a Riemannian manifold as a base space clear. 

 Our two examples are $\nu =F_{A} $, and $ \nu = D_{A} \Phi$. Here, $ F_{A}$, and $\Phi$ are assumed smooth in the domain $\Omega \subset \mathcal{R}^{n}$.
 In fact, we prove a generalization of the usual improved Kato inequalities, for arbitrary one and two forms, with error terms.
 Because a general two form $\boldsymbol{F}$ (unlike the curvature $ F_{A}$ does not satisfy either the Bianchi identity or the first Field equation, we expect the error terms to involve $D_{A}\boldsymbol{F} $ and $D^{*}_{A}\boldsymbol{F} $). Similarly, because a general one form $\boldsymbol{\theta}$ does not satisfy the second Field equation and 
 is not in the kernel of $D^{*}_{A}$, we expect the error terms to involve
$D_{A}\boldsymbol{\theta} $, and $ D^{*}_{A}\boldsymbol{\theta} $).

 We note that the constants are different for the one form and the two form
 inequalities.
 
 In this section we make use of the connection $\nabla$, locally associated with a 1-form $A$, in the usual way, and denoted by
 $\nabla^{A}$. We also make use of the associated local covariant derivative
 $\nabla_{A}$, and the associated local covariant exterior differential denoted
 by $D_{A}$. For a less cluttered notation, surpress the subscript $A$
 on the connection and the local covariant derivative, which will cause no 
 confusion because we are working in a local trivialization, so that the connection is the local covariant derivative.

 \begin{theorem}\label{T:pointwiseimprovedkatoinequalities}
 Let $\nabla$ be an arbitrary metric compatible connection on
 $ \Omega \times V$, where $ \Omega \subset \mathcal{R}^{n}$. If 
 $\boldsymbol{F} \colon \Omega \rightarrow
  V \otimes T^{*}(\Omega) \otimes T^{*}( \Omega )$ is an arbitrary smooth vector valued 2-form
 then,
 \begin{equation}\label{E:pointwiseimprovedkatoinequalitytwoformcase}
 \left( \frac{n}{n-1 }\right)
 \left  \vert d  \vert \boldsymbol{F}  \vert \vert^{2} \right) \leq
 \vert \nabla \boldsymbol{F} \vert^{2} + \vert D_{A} \boldsymbol{F} \vert^{2}
 + \vert D^{*}_{A} \boldsymbol{F} \vert^{2},
\end{equation}
and if $\boldsymbol{\theta} \colon  \Omega \rightarrow V \otimes T^{*} ( \Omega )$ is an arbitrary smooth vector valued one form,
\begin{equation}\label{E:pointwiseimprovedkatoinequalityoneformcase}
\left( \frac{n+1}{n }\right)
 \left  \vert d  \vert \mathcal{\boldsymbol{\theta}}  \vert \vert^{2} \right) \leq
 \vert \nabla \boldsymbol{\theta} \vert^{2} + \vert D_{A} \boldsymbol{\theta} \vert^{2}
 + \vert D^{*}_{A} \boldsymbol{\theta} \vert^{2}.
\end{equation}
\end{theorem}
\begin{proof}
First, we prove \eqref{E:pointwiseimprovedkatoinequalitytwoformcase}.
At any arbitrary point $p$ in the fiber $V$, we choose a local exponential
gauge, so that $A(p) =0$.
Choose an adapted orthonormal frame (inducing local adapted orthonormal coordinates), such that $d \vert \boldsymbol{F} \vert =
d_{1} \vert \boldsymbol{F} \vert = 
\frac{\partial}{\partial_{1}}(\vert \boldsymbol{F} \vert) d x^{1} $.
Note, that choosing such an orthonormal frame, still preserves the exponential gauge centered at $p$, because all we are doing is rotating the base coordinates, by a constant rotation at $p$.
Then using the standard Kato inequality, we have:
\begin{equation}\label{E:standardkatofortwoform}
	d \vert \boldsymbol{F} \vert =
d_{1} \vert \boldsymbol{F} \vert  \leq \vert \nabla_{1} \boldsymbol{F}\vert.
\end{equation}
Note, that in our adapted orthonormal coordinates, at the arbitrary point $p$
in the fiber--that is the center of our exponential coordinates, we have
$\nabla_{1} \boldsymbol{F} $ has the coordinate representation
$ \sum_{k,l} \frac{\partial}{\partial{x^{1}} }
(F_{k,l})$, because $A(p) =0$.
The idea of the proof is to make use of this formula for the coordinates of 
$\nabla_{1} \boldsymbol{F}$, in an expression resulting from replacing terms
in the coordinate representation of $D_{A} \boldsymbol{F}$ at $p$ by terms in the 
coordinate representation of $D^{*}_{A} \boldsymbol{F}$ at $p$. Then, we use that fact that $p$ is arbitrary.
First, we express $D_{A}(\boldsymbol{F}) $ at $p$ in components, with respect to our adapted coordinates. We note that at $p$, we have $D_{A}(\boldsymbol{F})= d (\boldsymbol{F})$.

Consider the individual components of $D_{A}\boldsymbol{F}(p) =d F(p)$. We can compute these explicitly in our local coordinates.
In particular, consider those component terms of the form
$\left( \frac{\partial}{\partial x^{1}} \underset{i \neq j \neq 1}{\boldsymbol{F}_{ij}} \right)$. We have ( for $i,j $ fixed):
\begin{equation}\label{E:typeonetermstwoformcase}
\left( \frac{\partial}{\partial x^{1}} \underset{i \neq j \neq 1}{\boldsymbol{F}_{ij}} \right) = 
\frac{\partial F_{1,j}}{\partial x^{i}} \pm 
\frac{\partial F_{1,i}}{\partial x^{j}} 
\pm (d \boldsymbol{F} )_{1,i,j}.
\end{equation}
Here, $ ( d \boldsymbol{F} )_{1,i,j}$ is the $dx^{1} \wedge dx^{i} \wedge dx^{j}$ component of $d \boldsymbol{F}$.

For each fixed pair $(i,j)$, on the left hand side of equation \eqref{E:typeonetermstwoformcase}, there are three terms on the right hand side of equation \eqref{E:typeonetermstwoformcase}.

Note that the $\pm$ parity of the terms on the right hand side of equation \eqref{E:typeonetermstwoformcase} is immaterial to our proof.

Consider the individual components of $D^{*}_{A} \boldsymbol{F}(p) = d^{*} \boldsymbol{F}(p)$.
In particular, those component terms of the form $\left( \frac{\partial}{\partial x^{1}} \underset{l  \neq 1}{\boldsymbol{F}_{1,l}} \right)$, are given
by:
\begin{equation}\label{E:typetwotermstwoformcase}
\left(
\frac{\partial}{\partial x^{1}} \underset{l \neq 1}{ \boldsymbol{F}_{1,l}}
\right)
=
\sum_{\substack{ s \neq 1 \\ s \neq l }}
\left[ 
\pm 
\left(
\frac{\partial}{\partial x^{s}} \boldsymbol{F}_{s,l}
\right)
\right]
\pm
( d^{*} \boldsymbol{F} )_{l}
\end{equation}
Here, $( d^{*} \boldsymbol{F} )_{l}$ is the $dx^{l}$ component of $d^*{*} \boldsymbol{F}$, where $l$ is the fixed $l$ on the left hand side of equation \eqref{E:typetwotermstwoformcase}.

Note that for each fixed $l$ on the left hand side of equation \eqref{E:typetwotermstwoformcase}, there are $n-2$ terms in the sum on the right hand side of \eqref{E:typetwotermstwoformcase}. This is because $ s$
takes the $n-2$ integer values $ \{s = 2, \dots, n \} - \{l \}$.
Thus, the right hand side of equation \eqref{E:typetwotermstwoformcase} has
$n-1$ terms.

Note that the $\pm$ parity of the terms on the right hand side of equation
\eqref{E:typetwotermstwoformcase} is immaterial to our proof.
We have
\begin{equation}\label{E:normofnablasub1Fidentity}
\vert \nabla_{1} \boldsymbol{F}(p) \vert^{2} =
\sum_{s \neq t }  \left \vert
\frac{\partial}{\partial x^{1}} \boldsymbol{F}_{st} \right \vert^{2}.
\end{equation}
Replacing each term on the right hand side of equation \eqref{E:normofnablasub1Fidentity} by terms not involving
$\frac{\partial}{\partial x^{1}}$, by using either equation \eqref{E:typeonetermstwoformcase} or equation \eqref{E:typetwotermstwoformcase}, we obtain an expression for
$ \vert \nabla_{1} \boldsymbol{F} \vert^{2}$, in which $\frac{\partial}{\partial x^{1}}$
does not appear.

Each such replacement has either $3$ or $n-1$ terms. Note, that, since $n \geq 4$, we have $3 \leq n-1$.

  Now, we use $ \left( \sum_{i =1}^{n-1} a_{i} \right)^{2}
  \leq 	(n-1) \left( \sum_{i =1}^{n-1} a_{i}^{2} \right)$
  .

Thus, we have: 
\begin{equation}\label{E:firstestimatefornormofnablasub1ofF}
	\vert \nabla_{1} \boldsymbol{F}(p) \vert^{2}
	\leq 	(n-1) 
	\left(
	\sum_{i=2}^{n} \vert \nabla_{i} \boldsymbol{ F}(p) \vert^{2}	
	\right)
	+ (n-1)\vert D_{A}^{*} \boldsymbol{F} \vert^{2} + (n-1) \vert D_{A} \boldsymbol{F} \vert^{2}.
\end{equation}
Adding $(n-1)\vert \nabla_{1} F(p) \vert^{2} $ to both sides of equation
\eqref{E:firstestimatefornormofnablasub1ofF} we obtain
\begin{equation}\label{E:secondestimatefornormofnablasub1ofF}
(n)\vert \nabla_{1} \boldsymbol{F}(p) \vert^{2}
\leq 
(n-1) 
	\left(
	\sum_{i=1}^{n} \vert \nabla_{i} \boldsymbol{F}(p) \vert^{2}	
	\right)
	+ (n-1) \vert D_{A}^{*} \boldsymbol{F} \vert^{2} + 
	(n-1) \vert D_{A} \boldsymbol{F} \vert^{2}.
\end{equation}
Dividing both sides of inequality \eqref{E:secondestimatefornormofnablasub1ofF} by $n-1$, we obtain
\begin{equation}\label{E:secondestimatefornormofnablasub1ofF}
\left(
\frac{n}{n-1}
\right)
\vert \nabla_{1} \boldsymbol{F}(p) \vert^{2}
\leq 
	\sum_{i=1}^{n} \vert \nabla_{i} \boldsymbol{F}(p) \vert^{2}	
	+ 
	\vert D_{A}^{*} \boldsymbol{F} \vert^{2} + 
	 \vert D_{A} \boldsymbol{F} \vert^{2}.
\end{equation}
Now, we use inequality \eqref{E:standardkatofortwoform}, in combination with
inequality \eqref{E:secondestimatefornormofnablasub1ofF}, and the fact that $p$ is arbitrary, to obtain:
\begin{equation}\label{E:thirdestimatefornormofnablasub1ofF}
 \left(
	\frac{n}{n-1}
\right)
 \vert \nabla_{1}  \vert\boldsymbol{F}(p) \vert	 \vert^{2} \leq
 \left(
	\sum_{i=1}^{n} \vert \nabla_{i} \boldsymbol{F}(p) \vert^{2}	
	\right)
	+
	\vert D_{A}^{*} \boldsymbol{F} \vert^{2} + 
	 \vert D_{A} \boldsymbol{F} \vert^{2}.
	\end{equation}

However, it follows from  formulae
(2.4) page 193, formula (2.12 ) page 194, and formula (2.13) page 194 of [2], that the inequality 
 \eqref{E:thirdestimatefornormofnablasub1ofF} is gauge invariant.
Since $p$ is arbitrary, inequality holds in any gauge and at any point in our local trivialization.
Note that inequality \eqref{E:thirdestimatefornormofnablasub1ofF}	
is inequality \eqref{E:pointwiseimprovedkatoinequalitytwoformcase}, and this completes the proof of inequality \eqref{E:pointwiseimprovedkatoinequalitytwoformcase}.

Now, in a similar way, we prove inequality \eqref{E:pointwiseimprovedkatoinequalityoneformcase}.
First, we prove \eqref{E:pointwiseimprovedkatoinequalitytwoformcase}.
At any arbitrary point $p$ in the fiber $V$, we choose a local exponential
gauge, centered at $p$ so that $A(p) =0$.
Choose an adapted orthonormal frame (inducing local adapted orthonormal coordinates), such that $d \vert \boldsymbol{\theta}(p) \vert =
d_{1} \vert \boldsymbol{\theta}(p) \vert = 
\frac{\partial}{\partial_{1}}(\vert \boldsymbol{\theta(p)	} \vert) d x^{1} $.
Note, that choosing such an orthonormal frame, still preserves the exponential gauge centered at $p$, because all we are doing is rotating the base coordinates, by a constant rotation at $p$.

Then using the standard Kato inequality, we have:
\begin{equation}\label{E:standardkatoforoneform}
	d \vert \boldsymbol{\theta}(p) \vert =
d_{1} \vert \boldsymbol{\theta}(p) \vert  \leq \vert \nabla_{1} \boldsymbol{\theta}(p)\vert .
\end{equation}
Note that in our adapted orthonormal coordinates at the arbitrary point $p$
in the fiber (that is the center of our exponential coordinates), we have
$\nabla_{1} \boldsymbol{\theta}(p) $ has the coordinate representation
$ \sum_{k,l} \frac{\partial}{\partial{x^{1}} }
(F_{k,l})$ because $A(p) =0$.
The idea of the proof is to make use of this formula for the coordinates of 
$\nabla_{1} \boldsymbol{\theta}$, in an expression resulting from replacing terms
in the coordinate representation of $D_{A} \boldsymbol{\theta}(p)$ at $p$ by terms in the 
coordinate representation of $D^{*}_{A} \boldsymbol{\theta}(p)$ at $p$. Then, we use that fact that $p$ is arbitrary.
First we express $D_{A}(\boldsymbol{\theta}) $ at $p$ in components, with respect to our adapted coordinates. We note that at $p$ we have $D_{A}(\boldsymbol{\theta})= d (\boldsymbol{\theta})$.

Consider the individual components of $D_{A}\boldsymbol{\theta}(p) =d \boldsymbol{\theta}(p)$. We can compute these explicitly in our local coordinates.
In particular, we consider coefficent terms that are the coefficents of $dx^{1} \wedge dx^{l} $. These terms satisfy

\begin{equation}\label{E:typeonetermsoneformcase}
\left( \frac{\partial}{\partial x^{1}}
 {\boldsymbol{\theta}_{l \neq 1}}(p)
  \right) =
 \pm 
 \frac{\partial\boldsymbol{\theta_{1}}}{\partial x^{l \neq 1}}(p)
 \pm (d \boldsymbol{\theta} )_{1,l} .
\end{equation}
Here, $( d \boldsymbol{\theta} )_{1,l}$ is the coefficent of $d \boldsymbol{\theta}$ corresponding to $dx^{1} \wedge d x^{l}$, where $l$
is the fixed $l$ on the left hand side of equation \eqref{E:typeonetermsoneformcase}.
There are two terms on the right hand side of equation \eqref{E:typeonetermsoneformcase}.
Note that, since $n \geq 4$, we have $2 \leq n$.

We also consider the individual components of $D^{*}_{A}(p) \boldsymbol{\theta}= d^{*} \boldsymbol{\theta}(p)$.
We have
\begin{equation}\label{E:typetwotermsoneformcase}
\left( \frac{\partial}{\partial x^{1}}
 {\boldsymbol{\theta}_{ 1}}(p)
  \right) =
  \pm \sum_{k=2}^{n} 
  \left( \frac{\partial}{\partial x^{k}}
 {\boldsymbol{\theta}_{ k}}(p)
 \right) + d^{*} (\boldsymbol{\theta}(p)).
\end{equation}

Note that $d^{*} (\boldsymbol{\theta}(p))$ is a zero form, and so it has no indices.
There are $n$ terms on the right hand side of equation \eqref{E:typetwotermsoneformcase}.

We have:
\begin{equation}\label{E:normofnablasuboneoftheta}
\vert \nabla_{1} \boldsymbol{\theta}(p) \vert^{2} = \vert d_{1} \boldsymbol{\theta}(p) \vert^{2}=
\left
\vert \left( \frac{\partial}{\partial x^{1}}
 {\boldsymbol{\theta}_{ 1}}(p)
 \right) \right  \vert^{2} 
 + \sum_{l=2}^{n}
 \left \vert \frac{\partial}{\partial x^{1}}
 \left \vert
 {\boldsymbol{\theta}_{l \neq 1}}(p)
  \right)
 \right \vert^{2} .
	\end{equation}
Replacing the first term on the right hand side of equation \eqref{E:normofnablasuboneoftheta} using 
\eqref{E:typetwotermsoneformcase},
and replacing the second term of the right hand side 
of equation \eqref{E:normofnablasuboneoftheta} using 
\eqref{E:typeonetermsoneformcase}, and using $\left( \sum_{i=1}^{n}
a_{i} \right)^{2}
\leq (n) \sum_{i=2}^{n} a_{i}^{2} $, we obtain

\begin{equation}\label{E:firstestimateofnormofnablasuboneoftheta}
\vert \nabla_{1} \boldsymbol{\theta}(p) \vert^{2} \leq 
n 
\left( \sum_{i=2}^{n} \vert \nabla_{i} \boldsymbol{\theta}(p) \vert^{2} \right)
+ n \vert D_{A}^{*} \boldsymbol{\theta}(p)	\vert^{2}
+n \vert D_{A} \boldsymbol{\theta}(p)	\vert^{2}.
\end{equation}
Adding $(n) \vert \nabla_{1} \boldsymbol{\theta}(p) \vert^{2} $ to 
both sides of inequality \eqref{E:firstestimateofnormofnablasuboneoftheta}
we obtain:
\begin{equation}\label{E:secondestimateofnormofnablasuboneoftheta}
(n+1) \vert \nabla_{1} \boldsymbol{\theta}(p) \vert^{2} \leq
(n) \vert \nabla \boldsymbol{\theta}(p) \vert^{2}+
 n \vert D_{A}^{*} \boldsymbol{\theta}(p)	\vert^{2}
+n \vert D_{A} \boldsymbol{\theta}(p)	\vert^{2}
\end{equation}
Dividing inequality \eqref{E:secondestimateofnormofnablasuboneoftheta}
by $n+1$ on both sides we obtain:
\begin{equation}\label{E:thirdestimateofnormofnablasuboneoftheta}
\vert \nabla_{1} \boldsymbol{\theta}(p) \vert^{2} \leq
\left(\frac{n}{n+1} \right) \vert \nabla \boldsymbol{\theta}(p) \vert^{2}+
 \vert D_{A}^{*} \boldsymbol{\theta}(p)	\vert^{2} +  \vert D_{A}^{*} \boldsymbol{\theta}(p)	\vert^{2}.
\end{equation}

Finally, we use the standard Kato inequalities \eqref{E:standardkatoforoneform}
together with inequality \eqref{E:thirdestimateofnormofnablasuboneoftheta}
to obtain
\begin{equation}\label{E:sharpversionofimprovedKatoineqforoneforms}
\vert \nabla \vert \boldsymbol{\theta}(p) \vert \vert^{2} 
\leq 
\left(\frac{n}{n+1} \right) \left(
 \vert \nabla \boldsymbol{\theta}(p) \vert^{2}+
 \vert D_{A}^{*} \boldsymbol{\theta}(p)	\vert^{2} +  \vert D_{A}^{*} \boldsymbol{\theta}(p)	\vert^{2}
 \right)
\end{equation}
at $p$ in our 
exponential gauge centered at $p$.
However, it follows from  formulae
(2.4) page 193, formula (2.12 ) page 194, and formula (2.13) page 194 of [2], that the inequality \eqref{E:sharpversionofimprovedKatoineqforoneforms} is gauge invariant.
Since $p$ is arbitrary, inequality holds in any gauge and at any point in our local trivialization.

Since inequality \eqref{E:sharpversionofimprovedKatoineqforoneforms} is inequality \eqref{E:pointwiseimprovedkatoinequalitytwoformcase}, we have proved
inequality \eqref{E:pointwiseimprovedkatoinequalitytwoformcase}.
This completes the proof of the theorem.
\end{proof}

These estimates are used in conjunction with the standard Weitzenbock formulae.
We remind the readers of these identities.
First, we have
\begin{theorem}\label{T:standardweitzenbockidentity}
\begin{equation}\label{E:standardweitzenbockidentity}
	\nabla_A^{*}\nabla_A V  = D^{*}_A D_A V  + D_A D^{*}_A V + S(F_A) V.
\end{equation}

Here, $ \nabla_A^{*}\nabla_A$ is the "rough Laplacian", $D_A$ is the exterior covariant differential, $D_A^{*}$ is the exterior covariant codifferential,
and  $S(F_A)$ is a bundle curvature term. There is no base curvature term, as we are assuming (locally) that the base manifold is an open domain of $\mathcal{R}^{n}$.
\end{theorem}
\begin{proof}
A good reference for this is section 3 of \cite{BL1981}. 
	\end{proof}

\begin{theorem}\label{T:weakWeitzenblockidentities} 
Let $(D_{A},\Phi)$ be a smooth solution of the field equations \eqref{E:eulerlagrangeeqs} in $\Omega$. Then
\begin{subequations}\label{E:smoothweitzenblockineq}
\begin{equation}\label{E:smoothweitzenblockineqforF}
\left( -\frac{1}{2} \right) 	\triangle \vert F_{A} \vert^{2}+
 \vert \nabla_{A} F_{A} \vert^{2} \leq \vert F_{A} \vert^{3} + \vert D_{A} \Phi \vert^{2} \vert F_{A} \vert + \vert \Phi \vert^{2} \vert F_{A} \vert^{2}  
 \end{equation}
 \begin{equation}\label{E:smoothweitzenblockineqforDphi}
 	\left(-\frac{1}{2} \right) \triangle \vert D_{A} \Phi \vert^{2} +
 \vert \nabla_{A} D_{A} \Phi \vert \leq  2 \vert F_{A} \vert \vert D_{A} \Phi \vert^{2} 
 + \vert \Phi \vert^{2} \vert D_{A} \Phi \vert^{2} +  \vert Q_{\Phi, \Phi }(\Phi) \vert \vert D_{A} \Phi \vert^{2} . 
 \end{equation}{}
\end{subequations}
\end{theorem}
\begin{proof}
We only give a sketch of the important points:
To prove \eqref{E:smoothweitzenblockineqforF}, we have:
\begin{align}\label{E:smoothweitzenblockineqforFandconsequences}
& D_{A} D^{*}_{A} F_{A} + D^{*}_{A} D_{A} F_{A} + [ F_{A},F_{A}] = \\ \notag
& \left(\frac{1}{2} \right)
D_{A} [ \Phi, D_{A} \Phi ] + 0 + [F_{A}, F_{A} ] =
[F_{A} , F_{A} ] + \left(\frac{1}{2} \right)[ D_{A} \Phi, D_{A} \Phi]  \\ \notag
& + \left(\frac{1}{2} \right)[ \Phi, [ F_{A}, \Phi ]] .
\end{align}
In equation \eqref{E:smoothweitzenblockineqforF} $D_{A}$ is the exterior covariant derivative, and $D_{A}^{*}$ is the exterior covariant codifferential.

Recall, that the inner product $< ,>$ of p-forms and q-forms with coefficents that are sections of an associated bundle is defined by taking the inner product of the section valued coefficents and producting with inner product of the form parts. Thus the Hodge star operator can be considered as acting on the form part alone.
Thus we have (using that the inner product on sections is a metric compatible with the connection $\nabla_{A}$ )
\begin{equation}
d<F_{A},F_{A} > = 2<F_{A},\nabla_{A} F_{A} >	
\end{equation}

\begin{equation}\label{E:laplacianofnormequation}
d*(<F_{A},\nabla_A F_{A}>) =  2< \nabla_A F_{A},\nabla_A F_{A}>  +  2<F_{A}, \nabla_A^{*} \nabla_A F_{A}> 
.
\end{equation}
Thus
\begin{equation}\label{E:secondlaplacianofnormequation}
d^{*}
 d(<F_{A},F_A >)
 = 
\left[ 2< \nabla_A F_{A},\nabla_A F_{A}>  +  2<F_{A}, \nabla_A^{*} \nabla_A F_{A}> \right]
.
\end{equation}

Now apply \ref{T:standardweitzenbockidentity} to the last term of \eqref{E:secondlaplacianofnormequation}, with $V = F_{A}$, noting that 
$D_{A} F_{A} =0$ by Bianchi's identity. Then apply \eqref{E:smoothweitzenblockineqforFandconsequences} to the result.
The first term of the right hand side of \eqref{E:secondlaplacianofnormequation} accounts for the second term of the lefthand side of \eqref{E:smoothweitzenblockineqforF}.

Similarly we prove \eqref{E:smoothweitzenblockineqforDphi}, by using
the identity
\begin{equation}\label{E:smoothweitzenbockidentityforPhi}
\nabla_{A}^{*} \nabla_{A} (D_{A} \Phi ) = (D_{A} D_{A}^{*} + D_{A}^{*} D_{A} ) (D_{A} \Phi ) + [F_{A}, D_{A} \Phi ]
	\end{equation}
 and the field equations \eqref{E:eulerlagrangeeqs}.
\end{proof}

\begin{corollary}\label{C:estimateoflaplacianofnusquared}
Let $\nu = ( F_{A}, D_{A}\Phi )$. Then
\begin{equation}\label{E:estimateoflaplacianofnusquared}
- \triangle (\vert \nu \vert^{2}) + \left(1 + \frac{1}{n} \right)
\vert d (\vert \nu \vert ) \vert^{2} \leq
3 \vert F_{A} \vert \vert \nu \vert^{2} + Q_{1} \vert \nu \vert^{2}
+ 4 \vert Q_{\Phi}(\Phi) \vert^{2} .
\end{equation}
Here $Q_{1} = 2max \left(\vert \Phi \vert^{2} + \vert Q_{\Phi, \Phi}( \Phi) \vert^{2} \right)$.
\end{corollary}
\begin{proof}
If we add equation \eqref{E:smoothweitzenblockineq} 
and equation \eqref{E:smoothweitzenblockineqforDphi}	, we get:
\begin{align}
&- \triangle (\vert \nu \vert^{2}) + \vert \nabla_{A} \nu \vert^{2}
\leq
\vert F_{A} \vert^{3}+ 3 \vert F_{A} \vert \vert 	D_{A} \Phi \vert^{2}
+ \vert \Phi \vert^{2}
\left(
\vert F_{A} \vert^{2} + \vert D_{A} \Phi \vert^{2}
\right)\\
&+  \vert 	Q_{\Phi,\Phi} (\Phi) \vert \vert	 D_{A} \Phi \vert^{2} 
\leq
3 \vert F_{A} \vert \vert \nu \vert^{2} +
\left(
\vert \Phi \vert^{2}
 +  \vert	Q_{\Phi,\Phi}(\Phi) \vert \vert \nu \vert^{2}
\right) .
\end{align}
From theorem \eqref{T:pointwiseimprovedkatoinequalities}, we get:
\begin{align}
&\left(
1 + \frac{1}{n}
\right)
 \vert d (\vert \nu \vert) \vert^{2}
 \leq
 \vert \nabla_{A} \nu \vert^{2}+
 \vert D_{A}^{*} F_{A} \vert^{2} +
 \vert D_{A}^{*} D_{A} \Phi \vert^{2}+
 \vert D_{A} D_{A} \Phi \vert^{2} \\
 &\leq 
 \vert 	\nabla_{A} \nu \vert^{2}+
 \vert \Phi \vert^{2} \vert \nu \vert^{2} +
 \vert Q_{\Phi}(\Phi) \vert^{2}.
\end{align}
Putting this inequalities together gives the result.
\end{proof}

\section{Bounds on the Solutions of an Elliptic Inequality: The Smooth Case }
\label{S:boundssmoothcase}
In this section we prove that a smooth function $f$ which satisfies an elliptic inequality in $\Omega$ is bounded in the interior of $\Omega$ in terms of its $X^{2}(\Omega)$ norm. This result, which is weaker than the result of  section \ref{S:Boundsthesingularcase}, can be used to prove that the limit of smooth solutions is smooth on the complement of a set of finite $n-4$
Hausdorff dimension. Also, it is a warmup for section \ref{S:Boundsthesingularcase}.

We prove this result for $\Omega_{1} \subset \Omega_{4} $, where $\Omega_{l} =
[-l, l ]^{n}$. By the results on scaling and monotonicity in section \ref{S:ScalingandmaxPHI}, Appendix A and section \ref{S:monotonicityformulae}, it can be modified for arbitrary domains.

We use the notation of Appendix A, where $X^{k} = M^{ \frac{n}{k}, \frac{4}{k} }$ for the Morrey Space with integration power
$\frac{n}{k}$ and scaling power $\frac{4}{k} $. The formulae are particularly simple in this notation.
\begin{theorem}\label{T:smoothsubsolutionmorreyestimate}
Let $u > 0$, and $f \geq 0$ be smooth functions in $\Omega_{4}$
with $ f \in X^{2}(\Omega_{4} )$.
\begin{equation}\label{E:semilinearsubequation}
-\triangle f + \alpha \left( \frac{\vert df \vert^{2}}{f} \right) - u f \leq Q_{1} f
.
\end{equation}
 Then there exist $\eta_{k}$ and $K_{k}$, depending on $\alpha >0 $ and $0 <k \leq 1$,
such that if $\Vert u \Vert_{X^{2}(\Omega_{4})} \leq \eta_{k}$,
\begin{equation}\label{E:smoothsubsolutionmorreyestimate}
\Vert f \Vert_{X^{k} (\Omega_{1}) } \leq K_{k} \Vert f \Vert_{X^{2} (\Omega_{4}) }
.
\end{equation}
Here $\eta_{k}$, and $K_{k}$ depend on the norm of the inversion of $\triangle$
on $X^{k'+ 2}$ and the norm of $X^{k'} \subset X^{2+k'}_{2}$, where $k \leq k' \leq 1$.
\end{theorem}

In fact, Theorem \eqref{T:smoothsubsolutionmorreyestimate} is true without the condition $\alpha >0$. However, we prove it with this condition as a warm
up for the proof of Theorem \eqref{T:Morreyboundforsolsofellipticineqoffsingset} of Section \ref{S:Boundsthesingularcase}.

Recall that the norm that we use for $X^{k}_{\gamma}(\Omega_{l} )$ denotes the cutoff of the odd extension. For $\gamma >0$, this imposes Dirichlet boundary conditions, where the condition $\gamma>0$ is necessary.

The first Lemma is a straightforward computation. 

\begin{lemma}\label{L:morreyreversepoincareforsubsolutions}
Suppose $u,f >0$ are smooth, and
\begin{equation}\label{E:semilinearsubequationwithplusc}
- \triangle f + \alpha \left( \frac{\vert df \vert^{2}}{f} \right) \leq (u+Q_{1}) f.
\end{equation}
Then
\begin{equation}\label{E:secondreversepoincareinequality}
\left
\Vert \frac{\vert df \vert}
{f^{1/2}} 
\right 
\Vert^{2}_{X^{2}(\Omega_{3}) } 
\leq C_{1} 
\Vert f \Vert^{2}_{ X^{2} (\Omega_{4})} 
\left( Q_{1}+ \Vert u 
\Vert_{X{^2}(\Omega_{4} ) } \right).
\end{equation}
\end{lemma}

\begin{proof}
Let
\begin{equation}\label{E:definitionofPsi}
\Psi =
\begin{cases}
1, \quad \text{for} \quad &  t \leq 1 \\
0, \quad \text{for} \quad & t \geq 2
\end{cases}
\end{equation}
be a smooth cutoff function, and for arbitrary $y \in \Omega_{3}$, let
$\Psi_{r}(x) = \Psi( \frac{\vert x-y \vert}{r} )$, $r < 1$.
Multiply equation \eqref{E:semilinearsubequationwithplusc}, by $\Psi_{r}(x)$, integrate and move the term 
$\int \triangle  \Psi_{r}(x) f(x) \, (dx)^{n} = 
\int   \Psi_{r}(x)  \triangle f(x) \, (dx)^{n} $ to the right hand side.
This gives,
\begin{equation}\label{E:weakcaccipoliestimate}
\alpha \int \Psi_{r}(x) \frac{\vert df (x) \vert^{2}}{f(x)} \, (dx)^{n} \leq
\int 
\left[
\vert \triangle (\Psi_{r}(x) ) \vert	
+ \Psi_{r}(x) (u+c)
\right] f
 \, (dx)^{n}
\end{equation}
Now
\begin{equation}\label{E:ballweakcaccipoliestimate}
\alpha \int_{\vert x-y\vert \leq r }   \left[ \frac{\vert df(x) \vert^{2}} {f(x)}  \right] \, (dx)^{n} \leq
\alpha \int \Psi_{r}(x)  \left[ \frac{\vert df(x) \vert^{2}}
{f(x)}  \right] \, (dx)^{n}
 \end{equation}
\begin{equation}\label{E:pointwiseestimateoflapPhisubr}
\vert \triangle\Phi_{r}(x) \vert \leq C_{2} r^{-2}
\end{equation}

\begin{align}\label{E:integralsubestimateforf}
& \int \vert  \triangle \Psi_{r}(x) \vert f(x)
\, (dx)^{n} )^{\frac{1}{2} }
\leq
C_{2} r^{-2} 
\left( \int_{\vert x-y \vert \leq 2r} f(x)^{2} \, (dx)^{n} \right)^{\frac{1}{2} }\left( \int_{\vert x-y \vert \leq 2r} 1 \, (dx)^{n} \right)^{\frac{1}{2}}\\ \notag
& \leq C_{3} r^{n-4} \Vert f \Vert_{X^{2}(\Omega_{4})}  
\end{align}
and
\begin{equation}
\int  \Psi_{r} (u+ Q_{1}) f \leq  (2r)^{n-4} \Vert u+Q_{1}  \Vert_{X^{2}(\Omega_{4})}
 \Vert f \Vert_{X^{2}(\Omega_{4})}.
\end{equation}

Putting this all together gives:
\begin{equation}
\alpha \int_{\vert x-y \vert \leq r}
\frac{\vert df \vert^{2}}{f}
\leq C_{3} ( r^{n-4} + (2r)^{n-4}\Vert u+Q_{1}  \Vert_{X^{2}(\Omega_{4})} )
(\Vert f \Vert_{X^{2}(\Omega_{4})} ).
\end{equation}
By adjusting the constants we get the required estimate.
Because $\Vert u \Vert_{X^{2}(\Omega_{4})}$ is already small, we can absorb it
in another constant.
\end{proof}

The next step in the proof of theorem \ref{T:smoothsubsolutionmorreyestimate}
is to bound $\Vert f \Vert_{X^{2}(\Omega_{4})} $.

Choose another smooth test function,
\begin{equation}
\hat{\psi}(x) =
\begin{cases}
0  & \quad \text{for} \quad x \in \mathcal{R}^{n} - \Omega_{3} \\
1. &  \quad \text{for} \quad x \in \Omega_{2}.	
\end{cases}
\end{equation}

According to theorem \ref{T:laplacianinversiononmorreyspaces}, in Appendix A, if $u$ is sufficently small, we can solve
\begin{equation}\label{E:semilinearlaplaceequationforphihat}
- \triangle \hat{\phi} -u \hat{\phi} =
Q_{1} \hat{\psi} f  - 2 d\hat{\psi} d f -
[ \triangle (\hat{\psi} ) ] f
\end{equation}
for $\hat{\psi} \in X^{3}_{2} (\Omega_{3} )$, if we can get an estimate of
the right hand side of equation \eqref{E:semilinearlaplaceequationforphihat}
in $X^{3}(\Omega_{3} )$. Since, $f \in X^{2} (\Omega_{3} ) \subset X^{3} (\Omega_{3} )$, the first and third terms are fine.
But, $df = \left(
\frac{df}
{f^{\frac{1}{2}}}
 \right)
  (f^{\frac{1}{2}})$.
  According to lemma \eqref{L:morreyreversepoincareforsubsolutions},
  $\frac{df}
{f^{\frac{1}{2}}}$ can be estimated in $X^{2}(\Omega_{3})$.
  Equation \eqref{E:powerlaw} of Appendix A
 shows that $ \Vert f^{\frac{1}{2}} \Vert_{X^{1}(\Omega_{3})} \leq 
\Vert f \Vert_{X^{2}(\Omega_{3})}^{\frac{1}{2}}$.
Moreover $X^{2}(\Omega_{3}) \otimes X^{1}(\Omega_{3}) \hookrightarrow
X^{3}(\Omega_{3})$ by multiplication. Hence the right hand side can be estimated in $X^{3}(\Omega_{3})$ by
\begin{equation}
( c + max \vert \triangle \hat{\psi} \vert ) \Vert f \Vert_{X^{2}(\Omega_{3})}
+ C_{4} \Vert f \Vert_{X^{3}(\Omega_{4})}.
\end{equation}
Hence,
\begin{equation}\label{E:morreynormestimateofphihatintermsofmorreynormoff}
\Vert \hat{\phi} \Vert_{X^{1}(\Omega_{3})} 
\leq
C_{5} \Vert \hat{\phi} \Vert_{X^{3}_{2}(\Omega_{3})}
\leq
C_{6}
 \Vert f \Vert_{X^{2}(\Omega_{4})}
\end{equation}
 where we have used the norm of the Morrey-Sobolev embedding
$ X^{3}_{2}(\Omega_{3}) \hookrightarrow X^{1}(\Omega_{3})
$.
We omit the dependence on $u$, because the norm involved is small by assumption.

Now let $g = \hat{\psi}f -\hat{\phi}$.
\begin{equation}\label{E:subellipticinequalityforg}
- \triangle g - ug \leq 0.
\end{equation}
According to theorem \eqref{T:gnegativitycondition} of Appendix
B, 
, 
if $\Vert u \Vert_{X^{2}(\Omega_{3}) } $
is sufficently small, then 
$g \leq 0$, and $\hat{\psi} f \leq 	\hat{\phi}$.
Then inequality \eqref{E:morreynormestimateofphihatintermsofmorreynormoff}
immediately transfers to:
\begin{equation}
	\Vert \hat{\psi} f \Vert_{X^{1}(\Omega_{3})} \leq
	C_{7}
	\Vert  f \Vert_{X^{2}(\Omega_{4})}.
\end{equation}
Now we take a second cutoff function
\begin{equation}
\bar{\psi} =
\begin{cases}
0 \quad \text{if} \quad & x \in \mathcal{R}^{n} - \Omega_{2} \\
1 \quad \text{if} \quad & x \in \Omega_{1}
\end{cases}.
\end{equation}
Note that $d ( \hat{\psi} \bar{\psi} ) = \triangle (\hat{\psi} \bar{\psi} )
= 0 $, and $\hat{\psi} \bar{\psi} = \bar{\psi}$.
Now
\begin{equation}\label{E:ellipticequationforpsihatpsibar}
- \triangle (\bar{\psi} \hat{\phi}   ) -
u \bar{\psi} \hat{\phi} =
c \bar{\psi} f -2 d \bar{\psi} d \hat{\phi} - ( \triangle \bar{\psi}) \hat{\phi}.
\end{equation}
We have $\bar{\psi} f \in X^{1}(\Omega_{2})$, and $\triangle (\bar{\psi} )\hat{\phi}    \in X^{1}(\Omega_{2})$, but $2 d \bar{\psi} d \hat{\phi}$
is only estimated in $X^{3}_{1}(\Omega_{2}) \subset X^{2}(\Omega_{2})$.
We cannot invert $ -\triangle - u$ on $X^{2}$, no matter how small
$\Vert u \Vert_{X^{2}(\Omega_{2})}$ is.
However $X^{2}(\Omega_{2}) \subset X^{2+ k}(\Omega_{2})$, and for
$\Vert u \Vert_{X^{2}(\Omega_{2})} \leq \eta_{k}$, we can invert 
$-\triangle - u$ on $ X^{2+ k}(\Omega_{2})$.
Now we have an estimate on $\bar{\psi} \hat{\phi} \in 
 X^{2+k}_{2}(\Omega_{2})
 \subset X^{2}(\Omega_{2})
 $, which transfers to
$ f \in X^{k}(\Omega_{1})$.
Note, that this problem is not linear in $f$, but it scales linearly in $f$.
This allows us to fix the dependence in the conclusion as linear in
$ \Vert f \Vert_{X^{2}}(\Omega_{4})$.

\section{Bounds on the solution of an elliptic inequality: the singular case}
\label{S:Boundsthesingularcase}
This section is similar to section \ref{S:boundssmoothcase}, except that we allow singular sets in $\Omega_{4}$.
\begin{theorem}\label{T:Morreyboundforsolsofellipticineqoffsingset}
Let $u \geq 0$, and $f>0$, be smooth functions on $\Omega_{4} - \mathcal{S}$,
where $\mathcal{S}$ is a closed set of finite $n-4$ Hausdorff dimension.
If 
$f \in X^{2}(\Omega_{4}) $,
$ 0 < k <4$ and
\begin{equation}\label{E:singularsubinequality}
-\triangle f + \alpha \left( \frac{\vert df \vert }{f} \right) -uf \leq Q_{1}f
\end{equation}
then there exist constants $\eta_{k} > 0$, and $\kappa_{k} > 0$, such that
if $\Vert u \Vert_{X^{2}(\Omega_{4}) } < \eta_{k}$, then $f \in X^{k}(\Omega_{1} )$. Moreover

\begin{equation}\label{E:morreyboundonfsingularcase}
\Vert f \Vert_{X^{k}(\Omega_{1})} \leq \kappa_{k} 
\Vert f \Vert_{X^{2}(\Omega_{4}) }.
\end{equation} 
\end{theorem}

\begin{proof}
We now modify the proof of section \ref{S:boundssmoothcase} to account for the singular Set $\mathcal{S}$, where $\mathcal{S}$ is a closed set of finite $n-4$ Hausdorff dimension.
We observe that inequality \eqref{E:morreyboundonfsingularcase}, though not linear, scales linearly. So we may assume that $\Vert f \Vert_{X^{2} (\Omega_{4} )} =1$, and get bounds on $\Vert f \Vert_{X^{k} (\Omega_{1} )} $ as a constant.

In the body of this proof, we need the following lemma:
\begin{lemma}\label{L:singularweightedreversepoincareestimate}
Suppose $u \geq 0$, $f >0$, with $u,f \in X^{2}(\Omega_{4})$ smooth off of a closed set $\mathcal{S}$ of finite $n-4$ Hausdorff dimension. In addition $\alpha > 0$, and $x \in \Omega_{4} - \mathcal{S}$, we have,
\begin{equation}\label{E:rearrangedsingularsubinequality}
	-\triangle f + \alpha \left( \frac{\vert df \vert }{f} \right)  \leq (u+ Q_{1})f,
\end{equation}
then
$\frac{\vert df \vert }{f^{\frac{1}{2}} } \in X^{2}(\Omega_{3})$, and
\begin{equation}\label{E:singularweightedreversepoincareestimate}
\left \Vert \frac{\vert df \vert }{f^{\frac{1}{2}} } \right\Vert_{X^{2} (\Omega_{3})}^{2}
\leq C_{\alpha} \Vert f \Vert_{X^{2}(\Omega_{4})} 
\left(
1 + \Vert u \Vert_{X^{2}_{\Omega_{4} } }
\right).	
\end{equation}
\end{lemma}
We prove Lemma \ref{L:singularweightedreversepoincareestimate}.
\begin{proof}
Suppose that the test function $\mu \in C^{\infty}_{0}(\Omega_{4} )$.
We will show that
\begin{equation}\label{E:weaksubeqsingularsetting}
\int (-\triangle \mu ) f \, (dx)^{n} + 
\alpha \int \mu \left( \frac{\vert df \vert^{2}}{f} \right) 
\, (dx)^{n} \leq 
\int (u+ Q_{1}) \mu f \, (dx)^{n}.
\end{equation}
Given equation \eqref{E:weaksubeqsingularsetting}, the proof is exactly the same as that of Lemma \ref{L:morreyreversepoincareforsubsolutions}, if we set $
\Psi_{R} = \mu$.
Let $S \colon \Omega_{4} \rightarrow \mathcal{R}^{+}$ be a regularized distance function (\ref{D:definitionofs}) to the set $\mathcal{S}$.
See Definition \eqref{D:definitionofs} of  Appendix B.
Let $\Psi \colon [0, \infty) \rightarrow  [0, \infty)$ be a $C^{\infty}$
function such that $\Psi$ has bounded derivative and

\begin{equation}\label{E:definitionofPsi}
 \Psi(t) =
\begin{cases}
1, \quad &\text{if } \quad t > 2 \\
0. \quad &\text{if}. \quad t \leq 1
\end{cases}.
\end{equation}

We define:
\begin{equation}\label{E:definitionofbetasubepsilon}
\beta_{\epsilon} = 1- \Psi(\frac{s(x)}{\epsilon}		) =
\begin{cases}
1, \quad &\text{if } \quad s(x) > 2 \epsilon \\
0. \quad &\text{if}. \quad s(x) \leq \epsilon
\end{cases}.
\end{equation}

Note by the usual computation, and by the definition of the regularized distance function, we have:
\begin{gather}
\vert d \beta_{\epsilon} \vert \leq K  \epsilon^{-1} \label{E:derivativeboundonbetasubepsilon} \\
\vert \triangle \beta_{\epsilon} \vert \leq K\epsilon^{-2} \label{E:boundsonlaplacianofbetasubepsilon}.
\end{gather}
Here the size of $K$ is inconsequential for the proof except that it is independent of $\epsilon$.
Then, equation \eqref{E:weaksubeqsingularsetting} is true if we replace $\mu$
by $ \beta_{\epsilon} \mu \in 
 C^{\infty}_{0}
 (\Omega_{4} -\mathcal{S} )$.
 Now we compute:
 \begin{equation}\label{E:boundonlaplacianofbetasubepsilonmu}
 \triangle (\beta_{\epsilon} \mu) =
 \triangle( \mu) \beta_{\epsilon} + 2 grad(\mu) \bullet grad (\beta_{\epsilon} ) + \mu \triangle (\beta_{\epsilon})
 \end{equation}
  and
 \begin{equation}
 \vert 	 2 grad(\mu) \bullet grad (\beta_{\epsilon} ) + \mu \triangle (\beta_{\epsilon}) \vert \leq K(\mu) \epsilon^{-2}
 \end{equation} 
  where $ K(\mu)$ does depend on $\mu$.
 This then implies:
 \begin{align}
 &\int_{\Omega_{4} } \beta_{\epsilon} \left( (\triangle \mu) f + \mu
 \frac{\vert df \vert^{2}}{f^{\frac{1}{2}} } \right) \, (dx)^{n} 
 \leq \\ \notag
 & \int_{\Omega_{4}} \beta_{\epsilon} (u+Q_{1}) f \, (dx)^{n} +
 K(\mu) \epsilon^{-2} \int_{\vert s(x) \vert \leq 2 \epsilon } f \, (dx)^{n}.
 \end{align}
 But,
 \begin{align}
 \int_{\vert s(x) \vert \leq 2 \epsilon}  f
  \, (dx)^{n} 
& \leq 
 \left(
 \int_{\vert s(x) \vert \leq 2 \epsilon}  \vert f \vert^{2}
  \, (dx)^{n}
  \right)^{\frac{1}{2} }
  \left(
  \int_{\vert s(x) \vert \leq 2 \epsilon} 1^{2} \, (dx)^{n}
  \right)^{\frac{1}{2} } \\ \notag
  &  \leq  \left( K(\mathcal{S }) \epsilon^{4} \right)^{\frac{1}{2} }
  \left(
  \int_{\vert s(x) \vert \leq 2 \epsilon}  \vert f \vert^{2}
  \, (dx)^{n}
  \right)^{\frac{1}{2} } .
  \end{align}
  Here the volume of $\{ X \mid s(x) < 2 \epsilon \}$ has been estimated
  in Lemma \eqref{L:Hausdorffintegralestimate} of Appendix B.
  However, since $ lim_{\epsilon \downarrow 0} \left(
  \int_{\vert s(x) \vert \leq 2 \epsilon}  \vert f \vert^{2}
  \, (dx)^{n}
  \right) = 0$, the $\epsilon^{2}$ and the $\epsilon^{-2}$ cancel, and we have our result.
  Note that the constants that depend on $\mathcal{S}$ and $\mu$ disappear in the limit, as they are multiplied by a term that goes to zero.
  This finishes the proof of the lemma.
  \end{proof}

In the remaining part of the proof of Theorem \ref{T:Morreyboundforsolsofellipticineqoffsingset}, we again use $\alpha > 0$, and assume
without loss of generality that
 $\alpha < \frac{1}{2} 	$.
Multiply equation \eqref{E:singularsubinequality} by $f^{-\alpha}$, and use the fact that in $\Omega_{3} -\mathcal{S}$, we have:
\begin{equation}\label{E:alphaweightedlaplaceidentity}
f^{-\alpha} \left(- \triangle f + \alpha \left(\frac{\vert df \vert^{2}}{f} \right) \right) = -d^{*}(f^{\alpha} df)= - \left(\frac{1}{1-\alpha} \right)
(\triangle 
(f^{1-\alpha} ) ).
\end{equation}
Define $\overline{f} = f^{1- \alpha}$, and note that:
\begin{equation}\label{E:relativesizeoffbarandf}
\Vert \overline{f} \Vert_{X^{2(1-\alpha)}(\Omega)} \sim \Vert f \Vert_{X^{2}(\Omega) }
\end{equation}
from equation \eqref{E:powerlaw} of Appendix A.
Now, on $\Omega_{3} - \mathcal{S}$, we have:
\begin{equation}\label{E:sublapacianinequalityforfbar}
-\triangle \overline{f} -u \overline{f} \leq Q_{1} \overline{f},
\end{equation}
for $\overline{f} \in X^{2(1-\alpha)}(\Omega_{3} )$.
Now we proceed with the proof of Theorem \eqref{T:Morreyboundforsolsofellipticineqoffsingset}.

In the proof of Theorem \eqref{T:smoothsubsolutionmorreyestimate}, we also needed an estimate on:
\begin{equation}\label{E:estimateondoverlinef}
d \overline{f} = (1-\alpha) \left(\frac{df}{f^{\frac{1}{2}}} \right) 
\left( f^{-\alpha + \frac{1}{2} } \right) \in X^{3 - 2 \alpha}(\Omega_{3}).
\end{equation}
This follows from the multiplication law and from
$\frac{df}{f^{\frac{1}{2}}} \in X^{2}(\Omega_{3})$, as well as $f^{ \frac{1}{2} - \alpha} \in X^{2(\frac{1}{2} -\alpha)}(\Omega_{3})$.
Here assuming that $\Vert f \Vert_{X^{2}(\Omega_{4}) }= 1$ is invaluable, as we need not carry these powers around.
In carrying out the proof, we obtain that
\begin{equation}\label{E:subequationforphibar}
-\triangle \overline{\phi} - u \overline{\phi}= Q_{1} \hat{\Psi} \bar{f} - 2 d \hat{\Psi}
d \overline{f}- \hat{\Psi}\triangle( \overline{f})
\end{equation}
 can be solved for
$\overline{\phi} \in X^{3-2 \alpha}_{2}(\Omega_{3})$. We have that
\begin{equation}\label{E:laplaciansubeqguationforgbar}
-\triangle( \overline{g} ) - u \overline{g} \leq 0,
\end{equation}
 for $\overline{g} = \hat{\Psi} f - \overline{\phi} $.
Here $\overline{g} \in X^{2(1- \alpha)}
(\Omega_{3})$, 
so the hypotheses of Theorem \ref{T:negativegoffsingularset} of Appendix B are satisfied with $\overline{g}^{1+\gamma}=
 \overline{g}^{\frac{1}{1-\alpha} }\in  X^{2}(\Omega_{3}) \subset L^{2}(\Omega_{3})$.
 It follows that $\overline{g} \leq 0$, and that $\hat{\Psi} \hat{f} \in X^{1-2 \alpha}(\Omega_{3})$.
 In the next step, transfering estimates similarly to the proof of Theorem 8, we get $\overline{\Psi} \overline{\phi} \in X^{\bar{k}}(\Omega_{2})$ for arbitrary $\bar{k}$ with $\Vert u \Vert_{X^{2}(\Omega_{2}) } \leq \eta_{\hat{k}}$.
 But $\overline{f} \in X^{\bar{k}}(\Omega_{1})$ is equivalent to
 $f \in X^{\bar{k} (1-\alpha)}(\Omega_{1})$.
 This completes the proof of Theorem \ref{T:Morreyboundforsolsofellipticineqoffsingset}.
 \end{proof}
 
 The following is a Corollary of the above proof of Theorem \ref{T:Morreyboundforsolsofellipticineqoffsingset}:
 
 \begin{proposition}\label{P:boundonf}
 If $u \leq \lambda f$, and the hypothesis of Theorem 9 of Appendix B are satisfied, then we have a bound on
 $f(x)$ for $ x \in \Omega_{1}$.
 \end{proposition}
 \begin{proof}
 We follow the proof of Theorem \ref{T:Morreyboundforsolsofellipticineqoffsingset}, until, with $\overline{f} \leq \overline {\phi} \in \Omega_{2}$, we have $\overline{\phi} \in X^{3-2}_{2}(\Omega_{3} )$.
 We have:
 \begin{equation}\label{E:laplacianequationforphibarpsibar}
 -\triangle ( (\overline{\Psi})( \overline{\phi} )) = u ((\overline{\Psi})( \overline{\phi} )) + Q_{1} (\overline{\Psi})(\overline{f})-
 2(d \overline{\phi})(d \overline{\Psi}) - (\triangle(\overline{{\Psi}})) (\overline{\phi}).
 	\end{equation}
  Now, all the terms on the right are in $X^{2-2\alpha}(\Omega_{2})$, except for the term $(u \overline{\Psi})(\overline{\phi})$.
 Here $\overline{\phi} \in X^{1-2\alpha}(\Omega_{2})$ and
 $u \leq \lambda f \leq \lambda (\overline{f})^{\frac{1}{1-\alpha}}
 \in X^{ \frac{1-2 \alpha}{1- \alpha}}(\Omega_{2}) \subset X^{1}(\Omega_{2})$.
 By multiplication, we have $u(\overline{\Psi})(\overline{\phi}) \in 
 X^{2-2 \alpha}(\Omega_{2})$. Hence $(\overline{\Psi})(\overline{\phi}) \in 
 X^{2-2 \alpha}_{2}(\Omega_{2}) \subset L^{\infty}(\Omega_{2})$, and
 $\overline{f}(x) \leq \overline{\phi}(x) \leq K$ in $\Omega_{1}$.
 Keeping track of the explicit dependence of powers of  $\Vert f \Vert_{X^{2}}(\Omega_{4})$ in the final estimate is not straightforward.
  \end{proof}
 	
\section{Application to Yang-Mills-Higgs}\label{S:application to Yang Mills Higgs}
The simpler results of Section \ref{S:boundssmoothcase} apply directly to getting estimates on smooth solutions, and we do not go into that here.

First, we directly apply the results of Section \ref{S:Boundsthesingularcase}
to solutions of a Yang-Mills-Higgs system in a cube $\Omega_{4}= [-4,4]^{n}$.
An immediate corollary, using Appendix C, shows when solutions (with Hausdorff codimension 4 singular set), extend to smooth solutions in the interior of $[-1,1]^{n}$. Later, we show how this applies to the Yang-Mills-Higgs equations in arbitrary domains, and discuss how to verify the hypotheses.
\begin{theorem}\label{T:linfinityboundsonFandDphisingularsetting}
Let $( D_{A},  \Phi )$ be solutions to a Yang-Mills-Higgs system in 
$\Omega_{4} -\mathcal{S} $, where $\mathcal{S}$ is a closed set of finite $n-4$
dimensional Hausdorff measure. Let $ \upsilon = ( F_{A},  D_{A}\Phi )$. Assume that $\upsilon \in X^{2}(\Omega_{4} )$, and that $ \Phi \in L^{\infty}(\Omega_{4})$. 
\begin{equation}\label{E:definitionofQsub1}
Q_{1}  = \underset{[-4,4]^{n}}{sup}
( 2 \vert \Phi \vert + 
 \vert Q_{\Phi,\Phi}(\Phi) ) \vert
\end{equation}
\begin{equation}\label{E:definitionofQsub2}
Q_{2}^{2} = \underset{[-4,4]^{n}}{sup} 
\left(\frac{\vert  Q_{\Phi}(\Phi)\vert^{2}}{Q_{1}}
\right).
\end{equation}
If $F_{A} \in X^{2}(\Omega_{4})$ is sufficently small, then $ \vert \upsilon \vert  \in L^{\infty}(\Omega_{1})$.
We also have the explicit bound 
\begin{equation}\label{E:L2nboundonupsilon}
\Vert \upsilon \Vert_{L^{2n}(\Omega_{1})} \leq
C(Q_{1}) ( \Vert \upsilon \Vert_{X^{2}(\Omega_{4})} + Q^{2}_{2} ).
\end{equation}
\end{theorem} 

\begin{proof}
Using the improved Kato formula, and the Weitzenbock formulae, we have, from
Corollary \eqref{C:estimateoflaplacianofnusquared}
of Theorem 7 see (3.27).
\begin{equation}\label{E:subellipticinequalityforabsoluteupsilon}
-\triangle(\vert \upsilon \vert^{2}) + ( 1 + \frac{1}{n} )
\vert d  \vert \upsilon \vert \vert^{2} \leq
3 \vert F_{A} \vert \vert \upsilon \vert^{2} + Q_{1}(\vert  \upsilon \vert^{2}
+Q^{2}_{2} ).
\end{equation}
Let $  u = 3 \vert F_{A} \vert  $, and $ f^{2} =\vert F_{A} \vert^{2} 
+ Q^{2}_{2}$. Then
\begin{equation}\label{E:ellipticsubequationforfsquared}
-(\frac{1}{2}) \triangle (f^{2}) +(1+\frac{1}{n})( \vert df \vert^{2} \leq
u f^{2} + Q_{1} f^{2}.
\end{equation}
Divide by $f$ and use the fact that
\begin{equation}\label{E:geometricdefinitionoflaplacianoff}
f^{-1} d^{*}(f d f ) + \frac{\vert df \vert^{2}}{f}  = \triangle f
\end{equation}
to get
\begin{equation}\label{E:sublaplacianinequalityforf}
-\triangle f + \left(\frac{1}{n} \right)\left( \frac{\vert df \vert^{2}}{f} \right) -u f \leq Q_{1} f.
\end{equation}
If $u = 3 \vert F_{A} \vert \in X^{2}(\Omega_{4})$
 is sufficently small, we can apply Proposition \ref{P:boundonf} to get a bound on $f \in L^{\infty}(\Omega_{1})$.

To obtain the explicit bound
\begin{equation}\label{E:Ltwobnoundonf}
\Vert f \Vert_{L^{2n}(\Omega_{1})}
\leq C(Q_{1}) \Vert f \Vert_{X^{2}(\Omega_{4})}
\end{equation}
 we apply Theorem \ref{T:Morreyboundforsolsofellipticineqoffsingset}
 with $\frac{4}{k} = 2n$. This gives:
\begin{equation}\label{E:L2nestimateonupsilon}
\Vert \upsilon \Vert_{L^{2n}(\Omega_{1})}
\leq \Vert \upsilon \Vert_{X^{k}(\Omega_{1})}
\leq C(Q_{1}) \left(  \Vert \upsilon \Vert_{X^{2}(\Omega_{4})}
+ Q_{2} \right).
\end{equation}
\end{proof}

Now, we have a regularity result in small balls of $\Omega_{1}$.
 This regularity theorem is proved by combining Theorem
 \ref{T:linfinityboundsonFandDphisingularsetting}
 and Corollary \eqref{C:interiorregularityestimate} from Appendix C. 
\begin{corollary}\label{C:gaugeequivalencetosmoothconnection}
If $ (
 D_{\tilde{A}}
, \Phi 
)
$ 
 satisfy the hypothesis of Theorem \ref{T:linfinityboundsonFandDphisingularsetting}, with $\Vert F_{\tilde{A}} \Vert_{X^{2}(\Omega_{4})}$ sufficently small, then for each point $y \in \Omega_{1}$ there is a neighborhood $B_{y}(\delta) \subset \Omega_{1}$, with $\delta$ chosen 
 such that (6.14) below holds, 
such that, in $B_{y}(\delta) - \mathcal{S}$, $D_{\tilde{A
} }$ is smoothly gauge equivalent to an
exterior covariant derivative $d +A$ (corresponding to a connection 
$\nabla_{A}$). 
If $\upsilon = ( F_{A}, D_{A} \Phi )$, we have:
\begin{equation}\label{E:L2nestimateonupsilon}
\Vert \upsilon \Vert_{L^{2n}(\Omega_{1})} \leq C(Q_{1}) \left(
\Vert \upsilon \Vert_{X^{2}(\Omega_{4}) } + Q_{2} \right)
\end{equation}

\begin{equation}\label{E:coloumbcondition}
d^{*} A = 0
\end{equation}

\begin{equation}\label{E:LpestimateofAintermsofLpestimateofF}
\delta^{-1} \Vert A \Vert_{L^{2n}(B_{y}(\delta))} \leq C
\Vert F_{A} \Vert_{L^{2n}(B_{y}(\delta))}
\leq  C \Vert F_{A} \Vert_{L^{2n}(\Omega_{1})}
\end{equation}

\begin{equation}\label{E:smoothnessofAandphi}
A \, \text{and} \, \,\Phi  \, \text{are smooth on} \,
B_{y}(\delta).
\end{equation}
\end{corollary}

\begin{proof}
Condition \ref{E:L2nestimateonupsilon} is the conclusion of Theorem \ref{T:linfinityboundsonFandDphisingularsetting}.
In Theorem \ref{T:coulombgauge}
 Appendix C 
the local trivialization in which the Coulomb condition of \ref{E:coloumbcondition} holds is constructed.
Let $\epsilon >0$.
The size of the ball $B_{y}(\delta)$ is fixed so that
\begin{equation}\label{E:numericalboundontheL2nnormofFinadeltaball}
\Vert F_{\tilde{A}} \Vert_{L^{2n}(B_{y}(\delta) )} \leq
\Vert F_{\tilde{A}} \Vert_{L^{2n}(\Omega_{1})}  \leq
\delta^{-3} \epsilon.
\end{equation}
If we rescale to a unit ball we have the $L^{2n}$ norm of $F$ bounded by 
$\epsilon$ and we can apply theorem
\ref{T:coulombgauge} 
of Appendix C.
Then, equation \ref{E:coloumbcondition} is valid.
The inequality \ref{E:LpestimateofAintermsofLpestimateofF} is the rescaled version of the estimate in 
Corollary \ref{C:interiorregularityestimate}
Appendix C.

Now collect the information that we have from Appendix C, and the Euler-Lagrange equations to get the following identities and equations:
\begin{equation}\label{E:covariantexteriorderivativeofA}
d^{*}A=0	
\end{equation}\label{E:definitionofF}

\begin{equation}
dA + \left(\frac{1}{2}	\right)[A,A] =F_{A}	
\end{equation}

\begin{equation}\label{E:equationfordstarofF}
d^{*} F_{A} +[A, F_{A} ]=
[\Phi, [ d+A,\Phi]]
\end{equation}

\begin{equation}\label{E:equationforDPhi}
D\Phi = [d+A, \Phi ]
\end{equation}

\begin{equation}\label{E:equationforDstarofdPhi}
D^{*} \left( d \Phi \right) + [A, d \Phi] = Q_{\Phi}(\Phi).
\end{equation}

Rearrange the above equations so that:
\begin{equation}\label{E:laplacianofAequation}
\triangle A = (d^{*}d + d d^{*} )A =
L_{1}(A, dA, \Phi, d \Phi )
\end{equation}
and

\begin{equation}\label{E:laplacianofPhiequation}
\triangle \Phi = d^{*} d \Phi =L_{2}(A, dA, \Phi, d \Phi ).
\end{equation}
 Here the powers of $\{A, dA, \Phi, d \Phi \} $ in the expressions
$L_{1}$ and $L_{2}$
 are under control and 				
$Q$ is a smooth function of $\Phi$, which -- to begin with-- is in $L^{\infty}$. Standard bootstrap arguments then yield smoothness in the interior.
\end{proof}

To see how Theorem \ref{T:linfinityboundsonFandDphisingularsetting}
applies in a general domain, assume $( D_{\tilde{A}}, \Phi )$ satisfies a Yang-Mills-Higgs system in $\Omega - \mathcal{S}$, where $\mathcal{S}$ is a closed set of finite $n-4$ dimensional Hausdorff measure.
For a fixed $y \in \mathcal{R}^{n}$,
let $\Omega_{y, \delta} = \{ x \colon  x-y  \in [ -\delta, +\delta ]^{n} \  
\}  $.
Then, Theorem \ref{T:linfinityboundsonFandDphisingularsetting} translates into:

\begin{corollary}\label{C:scaledandtranslatedgaugeequivalentregularitytheorem}
Assume $( D_{\tilde{A}}, \Phi )$ satisfies a Yang-Mills-Higgs system in $\Omega - \mathcal{S}$, where $\mathcal{S}$ is a closed set of finite $n-4$ dimensional Hausdorff measure.
Let $\upsilon = ( F_{A}, D_{A} \Phi )$.
 
\begin{equation}\label{E:defintionofmodifiedQ1}
Q_{1} = \underset{\Omega}{sup} \left( \vert  Q(\Phi) \vert^{2} +\ \vert  Q_{\Phi,\Phi}(\Phi) \vert \right)
\end{equation}

\begin{equation}\label{E:definitionofmodifiedQ2}
Q_{2}^{2} =  \left( \underset{\Omega}{sup} \vert Q_{\Phi}(\Phi) \vert^{2}
	\right) Q_{1}
\end{equation}
Suppose $\upsilon \in X^{2}(\Omega)$, and $Q_{1} \delta^{2}$ as well as
$Q_{2} \delta^{2}$ (scales like the two form $\upsilon$) are bounded by a fixed constant. In addition suppose that $ F_{\tilde{A}} \in X^{2}(\Omega)$ has small enough $X^{2}(\Omega)$ norm (independent of the other constants).
If $\Omega_{y,\delta} \subset \Omega$, then $\delta^{2} \upsilon \in L^{\infty}(\Omega_{y, \delta})$ is bounded above by a constant, and $( D_{\tilde{A}}, \Phi )$
and $
(D_{\tilde{A}}, \Phi )$
are smoothly gauge equivalent on 
$\Omega_{y,\delta}- \mathcal{S} $
 to a smooth exterior covariant differential (corresponding to a smooth connection), and a smooth Higgs Field on $\Omega_{y,\delta}$, both of which extend smoothly across the singular set $ \mathcal{S} \cap \Omega_{y,\delta} $.
\end{corollary}
\begin{proof}
In rescaling to $\tilde{x} = \frac{x-y}{\delta}$, the constants rescale as already described. Note that $X^{2}$, and the ( $X^{2}$ norm ) are invariant,
when applied to a geometric quantity that scales like a two form. Some examples
of such a quantity are $F_{\tilde{A}}$ and $D_{\tilde{A}} \Phi$. Thus Corollary \ref{C:scaledandtranslatedgaugeequivalentregularitytheorem} is a restatement of Theorem \ref{T:linfinityboundsonFandDphisingularsetting} for cubes of arbitrary size.
\end{proof}

We indicated in Section \ref{S:ScalingandmaxPHI} how a bound on the maximum
of norm $\Phi$ can be obtained. This leads to bounds on the terms $Q_{1}$ and
$Q_{2}$.

We emphasize that there are important examples where $\Phi \in L^{2}(\Omega)$
cannot be bounded. Our theory only applies when this bound is available.

The same can be said for the bound on $\upsilon = ( F_{A}, D_{A} \Phi) \in X^{2}(\Omega)$. Bounds on $\upsilon = (F_{A}, D_{A} \Phi) \in X^{2}(\Omega)$
are available when it is a stationary solution of a Yang-Mills-Higgs system.

In many cases, limits of smooth solutions approach smooth solutions (in an appropriate topology) in $\Omega - \mathcal{S}$ (where $\mathcal{S}$ is a closed set of finite $n-4$ dimensional Hausdorff measure), but it is not clear that these limits are stable with respect to perturbation by smooth
diffeomorphisms unless they fix the singular set $\mathcal{S}$.

We only have the following:
\begin{lemma}\label{L:X2normboundsonstablesoltions}
Suppose that $( D_{A},  \Phi )$ is a smooth solution of a Yang-Mills-Higgs system,
with $Q>0$ on $\Omega - \mathcal{S}$, where $\mathcal{S}$ is a closed set of zero
$n$ dimensional Lebesque measure. Suppose, in addition that $  ( D_{A},  \Phi )$
is a stationary critical point of the Yang-Mills-Higgs Functional  $\mathcal{A} (D_{A},\Phi )$,
 with respect to all smooth diffeomorphisms that fix the boundary of
$ \Omega$. Let $\Omega_{y, \delta} = \{ x \mid \vert x-y \vert < \delta \}$,
 assume that $dist ( \Omega_{y, \delta}, \mathcal{R}^{n} - \Omega) \geq R \geq \delta$ as well as
 \begin{equation}\label{E:Energyboundbyssquared}
 \int_{\Omega} \vert F_{A} \vert^{2} + \vert D_{A} \Phi \vert^{2} + Q(\Phi)
 \, (dx)^{n} \leq S^{2}.
 \end{equation}
  Then $ (F_{A}, D_{A} \Phi) \in X^{2}(\Omega_{y, \delta} ) $, and 
 \begin{equation}\label{E:X2normboundonupsilon}
 \Vert (F_{A}, D_{A} \Phi) \Vert_{X^{2}(\Omega_{y, \delta} )} \leq 
 R^{\frac{n-4}{2}} S.
 \end{equation}
\end{lemma}

\begin{proof}
By the Monotonicity Theorem \ref{T:monotonicitytheorem} for $x \in \Omega_{y, \delta}$ we have:
\begin{align}\label{E:monotonicitybounds}
\rho^{n-4	}
\int_{B_{x}(\rho) } \, 
\vert F_{A} \vert^{2} + \vert D_{A} \Phi \vert^{2} + Q(\Phi) \, (dx)^{n}\\
\leq 
R^{n-4	}
\int_{B_{x}(R) } \, 
\vert F_{A} \vert^{2} + \vert D_{A} \Phi \vert^{2} + Q(\Phi) \, (dx)^{n}
\leq R^{n-4} S^{2}.
\end{align}
If we use the norm $\Vert \bullet 
\Vert^{''}_{X^{2} (\Omega_{\delta})}$ defined in equation \eqref{E:doubleprimenormoff}
of Appendix A, we obtain the required estimate.
\end{proof}	

Our final regularity theorem is a corollary of 
Lemma \ref{L:X2normboundsonstablesoltions}, and of
Corollary \ref{C:scaledandtranslatedgaugeequivalentregularitytheorem}.

\begin{theorem}\label{T:finalregularitytheorem}
Assume the hypothesis of Lemma \ref {L:X2normboundsonstablesoltions} (with $A$ replaced by $\tilde{A}$ and the hypothesis of Corollary
\ref{C:scaledandtranslatedgaugeequivalentregularitytheorem} except for the small $X^{2}(\Omega)$ norm condition on $F_{\tilde{A}}$ in Corollary
\ref{C:scaledandtranslatedgaugeequivalentregularitytheorem}.
Suppose $ \Omega_{0} \in \Omega$ with
$ dist(\Omega_{0}, \mathcal{R}^{n}) =R
$
where $R-S< \epsilon$.
If $ \epsilon$ is positive and sufficently small and if
$S^{2}$ in the hypothesis of
Lemma \ref {L:X2normboundsonstablesoltions}
 is sufficently small, then 
 $(D_{\tilde{A}},\Phi) $ are smoothly gauge equivalent to a  smooth  exterior differential (corresponding to a smooth connection) and a smooth Higgs Field on $\Omega_{0}-\mathcal{S}
 $, and both extend smoothly across the Singular 
 Set $ \mathcal{S} \cap \Omega_{0}$. 
\end{theorem}
\begin{proof}
Using Lemma \ref {L:X2normboundsonstablesoltions},
we obtain, when $S^{2}$ is sufficently small that we can apply Theorem \ref{T:linfinityboundsonFandDphisingularsetting} in its equivalent form of Corollary \ref{C:scaledandtranslatedgaugeequivalentregularitytheorem}.
	\end{proof}

\appendix[ {\bfseries Appendix A Morrey Spaces}]\label{A:morreyspaces}\newline

The function spaces which arise naturally from the monotonicity formulae are 
Morrey Spaces. We outline a few key properties of these spaces. We follow the discussion and notation of [10]
, which is useful for geometers.
\begin{definition}\label{D:morreyspaces}
Let $p \geq q >1$.
The Morrey Space $M^{p,q}$ is the space of measurable functions on $\mathcal{R}^{n}$, with finite $M^{p,q}$ norm, where the Morrey norm $M^{p,q}$ is defined by
\begin{equation}\label{E:definitionofmorreyspaces}
\Vert f \Vert_{M^{p,q}}= max_{{\substack{y \in \mathcal{R}^{n} \\ r > 0} } }
(r^{ n (\frac{1}{p} - \frac{1}{q}	) } )
\left( \int_{\vert x-y\vert \leq r} \vert f \vert^{q} (dx)^{n}
\right)^{\frac{1}{q} }
\end{equation}
\end{definition}
Here $p$ is the scaling power, $q$ is the power of integration, and $L^{p} \subset M^{p,q}$. These are the same spaces as defined by Adams [1], and there denoted by $L^{q, \lambda}$. In our notation, we have 
$L^{q, \lambda} =  M^{p,q}$, where $\lambda = \frac{nq}{p}$. We use the spaces
$M^{\frac{n}{k},\frac{4}{k}  }= 
L^{\frac{4}{k},4 }
 =X^{k} $.

The Morrey--Sobolev spaces are spaces of functions $M^{p,q}_ {\alpha}$, with $\alpha$ derivatives in $M^{p,q}$.
Our two basic facts are:
\begin{theorem}\label{T:multiplicationin morreyspaces}
The map $(f,g) \rightarrow fg$,
( where $f \in M^{p,q}$, 
and $g \in M^{p',q'}_{q'}$) has the property that 
\begin{equation}\label{E:morreymultiplicationmap}
M^{p,q} \bigotimes M^{p',q'} \rightarrow M^{p'',q''} \quad \text{for}  \quad  \frac{1}{p} +
\frac{1}{p'} = \frac{1}{p''} \quad \text{and} \quad \frac{1}{q} + \frac{1}{q'} =
\frac{1}{q''} 
\end{equation}
This specializes to $ X^{k} \bigotimes  X^{k'} \rightarrow X^{k+k' }$.
\end{theorem}
\begin{proof}
The proof is a simple application of Holder's inequality.
\end{proof}
\begin{theorem}\label{T:morreysobolevembeddingtheorem}
\begin{equation}\label{E:morreysobolevembeddingtheorem}
M^{p,q}_{\alpha} \subset M^{p',q'} \quad \text{ for} \quad n > \alpha p \quad \text{, where} \quad q \geq 1 \quad and  \quad \frac{1}{p'} = \frac{1}{p} - \frac{\alpha}{n} \quad 
\text{ and} \quad 
\frac{p}{p'} = \frac{q}{q'}
\end{equation}
\begin{equation}\label{E:specializtionofmorreysobolevembeddingtheorem}
X^{k}_{\alpha} \subset X^{k - \alpha } \quad \text{ for} \quad \alpha < k 
\end{equation}
and
\begin{equation}\label{E:morreyholderembeddingtheorem}
M^{p,q}_{\alpha} \subset C^{\gamma} \quad \text{ for} \quad \alpha p > k
\end{equation}
\end{theorem}

\begin{proof}
This is on page 43 of Adams [1].
\end{proof}

So far, our function spaces are defined on $\mathcal{R}^{n}$.
Because we have choice of a domain, it suffices to fix a domain $\Omega_{l}$ in
$\mathcal{R}^{n}$.
Following Tao--Tian [10], we fix $\Omega_{l} = [-l,l]^{n}$. We use two
extensions for $f$ defined on $\Omega_{l}$.

\begin{definition} [Odd Extension ]\label{D:oddextension}
\begin{equation}
\hat{f}(x) =f(x)  \quad \text{, for} \quad x \in \Omega_{l}  
\end{equation}
\begin{equation}
\hat{f}(lk e_{j} +x)=  - f(lk e_{j} -x) 
\quad \text{, for} \quad x \in \Omega_{l} 
\quad \text{ , $k$ odd}.
\end{equation}
\end{definition}

\begin{definition} [Even Extension ]\label{D:evenextension}
\begin{equation}
\bar{f}(x) = u(x)  \quad \text{, for} \quad x \in \Omega_{l}  
\end{equation}
\begin{equation}
\bar{f}(lk e_{j} +x)=  \bar{f}(l ke_{j} -x) 
\quad \text{, for} \quad x \in \Omega_{l} 
\quad \text{ , $k$ odd}.
\end{equation}
\end{definition}

We also fix a smooth cutoff function $\Psi$, with support in $[-2,2]$, that is one on $[-1,1]$.
Let $\Psi_{l}(x) = \Phi(\frac{x}{l} ) $.
\begin{definition}[Morrey Norm Extension]\label{D:morreynormextension} 
\begin{equation}\label{E:morreynormextension} 
\Vert f \Vert_{M^{p,q}(\Omega_{l}) } 
= \Vert \Psi_{l} \hat{f} \Vert_{M^{p,q}} =
\Vert \Psi \bar{f} \Vert_{M^{p,q}}.
\end{equation}		
\end{definition}

\begin{definition}[Morrey--Sobolev  Norm Extension]\label{D:morreysobolevnormextension} 
\begin{equation}\label{E:morreysobolevnormextension} 
\Vert f \Vert_{M^{p,q}_{\alpha}(\Omega_{l}) }= 
= \Vert \Psi_{l} \overline{f} \Vert_{M^{p,q}_{l}}. 
\end{equation}		
\end{definition}

We note that Definition \ref{D:morreysobolevnormextension} is only useful for
Dirichlet boundary conditions.

The even extension is useful for Neumann boundary value problems, but we
only need the even extension for $u \in X^{2}(\Omega_{l})$.

We only use estimates on $\Omega_{l} = [-l,+l]^{n}$, for $1 \leq l \leq	 4$,
and by dilation arguments, the constants for $l=1$ differ from the constants 
for $1 \leq l \leq 4$ by fixed constants in the scale.

Some remarks about equivalent norms are in order:
We recall Definition \eqref{D:morreynormextension}.
We may also use:

\begin{definition}
\begin{equation}\label{E:primenormoff}
\Vert f \Vert_{M^{p,q}(\Omega_{l})}' = \underset{ \underset{x \in \Omega_{l}}{ r \leq l}}{max}
\left(
r
\right)^{\frac{n}{p} -\frac{n}{q}} 
\left(
\int_{\vert x-y \vert \leq r} \vert  \hat{f}(y) \vert^{q} \, (dy)^{n}
\right)^{\frac{1}{q}}	
\end{equation}
\end{definition}
and

\begin{definition}
\begin{equation}\label{E:doubleprimenormoff}
\Vert f \Vert_{M^{p,q}(\Omega_{l})}'' = \underset{ \underset{x,y \in \Omega_{l}}{ r \leq 1}}{max}
\left(
r
\right)^{\frac{n}{p} -\frac{n}{q}} 
\left(
\int_{\underset{\vert x-y \vert \leq r}
{x \in \Omega_{l}}} \vert  f(y) \vert^{q} \, (dy)^{n}
\right)^{\frac{1}{q}}.	
\end{equation}
\end{definition}

There are combinatorial constants involved in the comparison estimates.
But, each of the following are easily proved to be valid in any of these norms:

\textbf{Restriction}:
\begin{equation}\label{E:restriction}
M^{p,q}(\Omega_{l}) \hookrightarrow M^{p,q}(\Omega_{l'}) \quad \text{, where}
\quad l' \leq l
\end{equation}

\textbf{Subspaces}:
\begin{equation}\label{E:subspaces}
M^{p,q}(\Omega_{l}) \hookrightarrow M^{p',q'}(\Omega_{l'}) \quad \text{, where}
\quad p' \leq p \, ,
\quad
q' \leq q
\end{equation}

\textbf{Multiplication}:
\begin{equation}\label{E:multiplication}
M^{p',q'}(\Omega_{l})
\otimes
M^{p'',q''}(\Omega_{l})
 \hookrightarrow M^{p,q}(\Omega_{l'}) \, \text{, where}
\quad \frac{1}{p'} + \frac{1}{p''}  \leq 
\frac{1}{p}
 \, ,
\quad
\quad \frac{1}{q'} + \frac{1}{q''}  \leq 
\frac{1}{q}.
\end{equation}

\textbf{Power Law} (for ' and '' norms):
\begin{equation}\label{E:powerlaw}
\Vert f^{\alpha} \Vert_{M^{p,q}(\Omega_{l})} \leq 
\left(\Vert f \Vert_{M^{\alpha p,\alpha q}(\Omega_{l})}
\right)^{\alpha}
 \quad \text{, where} \quad
\alpha > 0
\end{equation}

\textbf {Invertibility of $\triangle$ and $\triangle-cu$ on Morrey Spaces}

First, we have an invertibility result for the Laplace operator on Morrey Spaces defined on $\Omega$

\begin{theorem}\label{T:laplacianinversiononmorreyspaces}
$\triangle \colon M^{p,q}_{2}(\Omega) \rightarrow M^{p,q}(\Omega) $ is 
invertible.
\end{theorem}
\begin{proof}
We prove this for $\Omega = [-1,1]$. The proof for arbitrary $\Omega_{l}$
is obtained by scaling.
First, we solve $ \triangle f =g$ in $\Omega$, with Dirichlet boundary conditions. Then, $ \triangle \hat{f} = \hat{g}$ in $\mathcal{R}^{n}$.
Choose
\begin{equation}
\psi_{3}(x) =
\begin{cases}
0 \quad \text{for,}	\quad & x \notin \Omega_{4} \\
1 \quad \text{for,}	\quad & x \in \Omega_{3}
\end{cases}.
\end{equation}
Then
\begin{equation}
\triangle ( \psi_{3} \hat{f} ) = 
\psi_{3} \hat{g} +
(\triangle( \psi_{3})) \hat{f}
+ 2 (d \psi_{3}) \bullet (d \hat{f}).
\end{equation}
Now let
\begin{equation}
 \psi_{3} \hat{f} = f_{1} + f_{2}
\end{equation}
where
\begin{align}
& \triangle f_{1} = \psi_{3} \hat{f} \in M^{p,q} \\
& \triangle f_{2} = 2 d\psi_{3} \bullet d \hat{f} +
( \triangle \psi_{3} ) \hat{f} \in L^{q}_{1}.
\end{align}
Here $f_{1} \in M^{p,q}_{2}$ by the invertibility of $\triangle$ on
$M^{p,q}$, ( see theorem
8.1
of Adams [1]).
We have $f_{2}$ in an appropriate Sobolev space on $\mathcal{R}^{n}$.
But $\triangle f_{2} = 0 $ in $\Omega_{3}$. Thus by elliptic regularity
$f_{2} \mid \Omega_{2} \in C^{\infty}( \Omega_{2} )$.
Hence,
\begin{align}
&\Vert f \Vert_{M^{p,q}_{2}(\Omega )} =
\Vert \psi_{3} \hat{f} \Vert_{M^{p,q}_{2}} \leq
\Vert \psi_{3} f_{1} \Vert_{M^{p,q}_{2}} +
\Vert \psi_{3} f_{2} \Vert_{M^{p,q}_{2}} \\
&\leq
C_{1} \Vert f_{1} \Vert_{M^{p,q}_{2} } +
C_{2} \Vert f_{2} 
\mid \Omega_{2} 
\Vert_{C^{\infty}(\Omega_{2})}.
\end{align}
\end{proof}

We are next interested in the properties of $\triangle - u \colon M^{p,q}_{2}(\Omega) \rightarrow M^{p,q}(\Omega)$.
Again, to define the operator, we note that $\triangle - \bar{u}$, where $\bar{u}$ is the even extension of $u$ to $\mathcal{R}^{n}$, has the desired
properties.
Consider
\begin{equation}\label{E:laplacianwithlowerordertermequation}
( \triangle - \bar{u} ) \hat{f} = \hat{g}.
\end{equation}

We need only to estimate the norm of $\psi \bar{u} \hat{f}$ in $M^{p}_{q}$ in terms of $\psi \hat{f}$ in $M^{p}_{q,2}$.
\begin{theorem}\label{T:invertibilityoflaplacianwithlowerorderterm}
Let $2<k<4$.
If $u \in M^{\frac{n}{2},2} (\Omega) = X^{2}(\Omega)$ is sufficently small then

\begin{equation}\label{E:invertibilityoflaplacianwithlowerorderterm}	
\triangle -u \colon X^{k}_{2}(\Omega) \rightarrow X^{k}(\Omega)
\end{equation}
is invertible.
\end{theorem}
\begin{proof}
This follows from $X^{k}_{2} \hookrightarrow X^{k-2}$, and $X^{k-2} \bigotimes X^{2} \hookrightarrow X^{k} $. See Theorem \ref{T:multiplicationin morreyspaces}
Note that our spaces assume Dirichlet boundary data.
\end{proof}

\appendix[{ \bfseries Appendix B Eigenvalues and the Maximum Principle}]\label{A:AppendixB} \newline
The goal of this appendix is to prove a maximum principle for $ -\triangle -u $, when $u \in X^{2}(\Omega_{l} )$ is small on $\Omega_{l} = [-l,l]^{n}$,
 $1 \leq l \leq 4$. Since the constants change by fixed amounts, without loss of generality, we can assume $\Omega = [-1,1 ]^{n}$.
 \begin{theorem}\label{T:morreypoincareineq}
 There exists a constant $\lambda$, (depending on the norm of $\triangle^{-1}$ on
 $X^{2}(\Omega)$ and on the constants in the Morrey-Sobolev embedding $X^{3}_{2} (\Omega) \subseteq  X^{1}( \Omega)$  such that
 \begin{equation}\label{E:morreypoincareineq}
 \lambda \int_{\Omega} u \phi^{2} \,  (dx)^{n} \leq \Vert u \Vert_{X^{2}(\Omega) } \int_{\Omega} 
\vert d \vert \phi \vert \vert^{2} \, (dx)^{n}
 \end{equation}
  for $\phi \in L^{2}_{1,0} (\Omega)$.
 \end{theorem}
 \begin{proof}
 It is sufficent to prove this for $\phi$ smooth, since  since smooth functions are dense in $L^{2}_{1,0} (\Omega)$. Fix such a $\phi_{0}$.
 Then, choose $\rho$ so that
 \begin{equation}\label{E:poincaremorreystringofineqs}
 \int_{\Omega} \vert d \phi_{0} \vert^{2} \, (dx)^{n} \leq
 \rho \int_{\Omega} u \phi_{0}^{2} \, (dx)^{n} \leq 
 \rho \int_{\Omega} u_{c} \phi_{0}^{2} \, (dx)^{n} +
  \rho \int_{\Omega} ( u -u_{c} ) (max \phi_{0}^{2} )\, (dx)^{n} 
 \end{equation}
 Here
 \begin{equation}
 u_{c}= \begin{cases}
 u, & \text{ if } u \leq c \\
 c & \text{if } u \geq c
 \end{cases}
 \end{equation}
 and $\lim_
 {c \rightarrow \infty }
  \int_{\Omega}( u-u_{c}) \, (dx)^{n} = 0$.
  Note that
  \begin{equation}\label{E:cutoffsmorreynorm}
  \Vert u_{c} \Vert_{X^{2}} \leq \Vert u \Vert_{X^{2}}
  \end{equation}
  Minimize $\int_{\Omega} \vert d \phi \vert^{2} \, (dx)^{n}$ subject to the constraint $ \int_{\Omega} u_{c} \phi^{2} \, (dx)^{n} =1$, for $\phi \in L^{2}_{1,0}(\Omega)$. Since $L^{2}_{1,0}(\Omega) \subset L^{2}(\Omega)$ is compact and $u_{c} \leq c$ we get an eigenvalue $\rho_{c}$ and an eigenfunction $ \phi_{c}$ in $L^{p}_{2,0}$ for all $p$,  such that	
  \begin{subequations}\label{E:eigenequationandpoincare}
  \begin{equation}\label{E:eigenequation}
  -\triangle \phi_{c} -\rho_{c} u_{c} \phi_{c} =0
  	\end{equation}
  \begin{equation}\label{E:eigenpoincareineq}
  \rho_{c} \int_{\Omega} u_{c} \phi^{2} \, (dx)^{n} \leq \int_{\Omega} \vert d \phi \vert^{2} \, (dx)^{n} 	
  \end{equation}
   for all $\phi \in L^{2}_{1,0}(\Omega)$
\end{subequations}
 But from \eqref{E:eigenequation}, we see that
 
 \begin{equation}\label{E:laplacianinequalities}
 \Vert \triangle \phi_{c} \Vert_{X^{3}
 (\Omega)} \leq	 \rho_{c} 
 \Vert u_{c} \phi_{c} \Vert_{X^{3}}(\Omega) 
 \leq \rho_{c} \Vert u_{c} \Vert_{X^{2}(\Omega)} \Vert \phi_{c}
 \Vert_{X^{1}}(\Omega).
  \end{equation}
  However
  \begin{equation}\label{E:mixedmorreyeigenvalueestimate}
\Vert  \phi_{c} \Vert_{X^{1}(\Omega) } \leq c_{1}
\Vert  \phi_{c} \Vert_{X^{3}_{2} (\Omega) } \leq
c_{1} c_{2} \Vert  \triangle \phi_{c} \Vert_{X^{3} (\Omega) }
\leq c_{1} c_{2} \rho_{c} 
\Vert u_{c}  \Vert_{X^{2}(\Omega)} \Vert \phi_{c} \Vert_{X^{1} }(\Omega).
\end{equation}
Hence 
\begin{equation}\label{E:lower1morreyinequality}
1 \leq c_{1}c_{2} \Vert u \Vert_{X^{2}(\Omega)} \rho_{c}.
\end{equation}
Use inequality \eqref{E:cutoffsmorreynorm}, inequality \eqref{E:eigenpoincareineq} and inequality \eqref{E:lower1morreyinequality}
to get:
\begin{align}
&\int_{\Omega} u \phi_{0}^{2} (dx)^{n} 
\leq
c_{1}c_{2} \Vert u \Vert_{X^{2}(\Omega)} \rho_{c} \int_{\Omega} 
u_{c} \phi_{0}^{2} (dx)^{n} +\\
\notag
&\int_{\Omega} (u-u_{c}) \phi_{0}^{2} (dx)^{n} 
\leq c_{1} c_{2} \Vert u \Vert_{X^{2}(\Omega)} \int_{\Omega} \vert d \phi_{0} \vert^{2} (dx)^{n} + \\ \notag
&\int_{\Omega)} (u -u_{c} ) (dx)^{n} 
\underset{x \in \Omega}{ max} \phi_{0}^{2}.
\end{align}
Since $\lim_{c \rightarrow \infty} \int_{\Omega} (u -u_{c}) (dx)^{n} =0$,
we get the result with $\lambda = \frac{1}{c_{1} c_{2}}$.
  \end{proof}
  
  The proof of the smooth theorem is immediate from 
  Theorem 16.
  \begin{theorem}\label{T:gnegativitycondition}
  If $(u,g)$ are smooth, there exists a constant $\eta$ depending on the norm of $\triangle^{-1}$ and on a Morrey-Sobolev embedding constant, such that if
  $\Vert u \Vert_{X^{2}(\Omega)} < \eta$,  $g=0$ on $\partial \Omega$ and
  \begin{equation}\label{E:gsubequation}
  - \triangle g - u g \leq 0,
   \end{equation}
   then $g \leq 0$
  \end{theorem}
  \begin{proof}
  	Let 
  	\begin{equation}\label{E:definitionofgplus}
  	g_{+}(x) =
  	\begin{cases}
   0 & \text{if}  \quad g(x) \leq 0 \\
  	g(x) & \text{if}  \quad  g(x) \geq 0
  	\end{cases}.
  	\end{equation}
  	Then from \eqref{E:gsubequation}
  	\begin{equation}\label{E:sublaplacianestimateongplus}
  	\int_{\Omega} \left( \vert d g_{+} \vert^{2} - u g_{+}^{2} \right) \, (dx)^{n} =
  	\int_{\Omega} (- \triangle g - u g ) g_{+} \, (dx)^{n} \leq 0.
  \end{equation}
  But,
  \begin{equation}\label{E:gplusweightedL2lowerboundonu}
  \int_{\Omega}  u g_{+}^{2} \, (dx)^{n} \geq 
  	\int_{\Omega} \vert d g_{+} \vert^{2} \, (dx)^{n}
  	\geq 
  	\left( \frac{\lambda}{\Vert u 
  	\Vert_{X^{2}
  	(\Omega) }}
  	\right)
  	\int_{\Omega}  u g_{+}^{2} \, (dx)^{n} 
  \end{equation}
   from  Theorem \ref{T:morreypoincareineq}.
   If
  $\lambda > \Vert u \Vert_{X^{2}(\Omega)}$
  , then
  \begin{equation}\label{E:integralugplusequalzeroeq}
  \int_{\Omega} u g_{+}^{2} \, (dx)^{n} =
  \int_{\Omega} \vert dg_{+} \vert^{2} \, (dx)^{n} =
  0.
  \end{equation}
  Hence $\lambda = \eta$ of  Theorem \ref{T:morreypoincareineq}.
  \end{proof}
  After this warm-up, we only need a few additional ideas to handle the singular case.
  \begin{definition}\label{D:definitionofs}
  If $\mathcal{S} \subset \Omega $ is a closed singular set, a regularized
  distance  function to $\mathcal{S}$ is a map $s \colon \Omega \rightarrow \mathcal{R}^{+}$, such that $s(x) =0 $, for $x \in \mathcal{S}$, is smooth and
  $s \colon \Omega - \mathcal{S} \rightarrow \mathcal{R}$, and
  \begin{equation}\label{E:steindistancebounds}
  c^{-1} (dist(x, \mathcal{S} ) )\leq s(x) \leq 
  c (dist(x, \mathcal{S} )).
   \end{equation}
   Furthermore the k-th derivative of $s$ satisfies
   \begin{equation}\label{E:derivativeboundsonsteindistance}
 \vert (\frac{ d}{dx_{i}} )^{k}
  s(x) \vert 
  \leq C_{k} (s(x))^{-k+1}  C_{k} (dist(x, \mathcal{S} )^{-k + 1} )
    \end{equation}
    on $\Omega - \mathcal{S}$.
  	\end{definition}
  	
  	The existence of this regularized distance function is a theorem of Stein [6] (Theorem 2 page 171) 
  	.
  	The following lemma follows from the definition of Hausdorff measure and a counting argument.
  	\begin{lemma}\label{L:Hausdorffintegralestimate}
  	If $\mathcal{S}\subset \Omega$ is a closed set of finite k-dimensional
  	Hausdorff measure, and $s \colon \Omega \rightarrow \mathcal{R}$ is a regularized distance function to $\mathcal{S}$, then
  	\begin{equation}\label{E:Hausdorffintegralestimate}
  	\int_{s(x) \leq r} 1 \leq  \bar{K} r^{n-k}.
  	\end{equation}
  	
  	\end{lemma}

  	\begin{theorem}\label{T:negativegoffsingularset}
  	Suppose $\mathcal{S} \subset \Omega$ is a closed set of finite $n-4 $ dimensional Hausdorff measure.
  	Let $g^{2 \beta} = g^{1 + \gamma} \in L^{2}(\Omega) $  (which defines $\beta $) for $\gamma >0$, where $u, g \in C^{\infty}(\Omega - \mathcal{S}) $ and $g=0$ on $\partial \Omega - \mathcal{S}$.
  	If $u \in X^{2}(\Omega)$ is sufficently small (depending on $\gamma >0$ )
  	and
  	\begin{equation}\label{E:secondgsublaplacianeq}
  	-\triangle g -u g \leq 0 ,
  	\end{equation}
  	then $g \leq 0$.
  	\end{theorem}
  	\begin{proof}
  		Let $\Psi$ be a smooth cutoff function with
  	\begin{equation}
  	\Psi(t) =
  	\begin{cases}
  	0 & \text{for} \quad t \leq 1 \\
  	1 & \text{for} \quad   t \geq 1	
  	\end{cases}.
\end{equation}
Assume $s$ is a regularized distance function, and define $\Psi_{R}(x) = \Psi(\frac{s(x)}{R} )$, for $R > 0$.
Let
\begin{equation}
g_{\epsilon} =
\begin{cases}
g- \epsilon & \text{for} \quad g \geq \epsilon \\
0          & \text{for} \quad  g \leq \epsilon
\end{cases}.
\end{equation}
Choose $\epsilon > 0$ such that $\epsilon$ is a regular value of $g$ on
\\$\Omega- \mathcal{S}$. We will take $\epsilon \rightarrow 0$, so that
\begin{equation}
g_{0} =
\begin{cases}
g & \text{for} \quad g \geq 0 \\
0 & \text{for} \quad g \leq 0
\end{cases}.
\end{equation}
We wish to prove $g_{0}=0$.
Now, we have on $\Omega$
\begin{equation}\label{E:gepsilonweightedsublaplacianineq}
- \Psi_{R}^{2} ( g_{\epsilon}^{\gamma} \triangle g -
u g_{\epsilon}^{\gamma} g ) \leq 0
 \end{equation}
 
 \begin{equation}
 u g_{\epsilon}^{\gamma} g =  u g_{\epsilon}^{2\beta} +
 \epsilon u g_{\epsilon}^{\gamma}.
 \end{equation}
 
 We also have:
 \begin{equation}
 \Psi_{R}
 ^{2}
  g_{\epsilon}^{\gamma} \triangle g =
 div (\Psi_{R}^{2} g_{\epsilon}^{\gamma} d g )-
 (\frac{1}{2 \beta})
 (d \Psi_{R}^{2}) g_{\epsilon}^{2 \beta} )
  -
 \left(\frac{\gamma}{\beta^{2}} \right)
 \Psi_{R}^{2} \vert d g_{\epsilon}^{\beta} \vert^{2}
 +(\frac{1}{2 \beta}) (\triangle \Psi_{R}^{2})
 g_{\epsilon}^{2 \beta}.
 \end{equation}
 
 Continue, to get:
 \begin{equation}\label{E:differentialinequality}
 \Psi_{R}^{2} \vert d g_{\epsilon} \vert^{2} =
 \vert d( \Psi_{R} g_{\epsilon}^{\beta}) -	
 (d \Psi_{R}) g_{\epsilon}^{\beta} \vert^{2} \geq
 \left( \frac{1}{2} \vert d( \Psi_{R} g_{\epsilon}^{\beta}) \vert^{2} -
 \vert d \Psi_{R} 
 \vert^{2}
  g_{\epsilon}^{2\beta}
 \right).
 \end{equation}
 
 Putting \eqref{E:gepsilonweightedsublaplacianineq} to \eqref{E:differentialinequality}
   together, we obtain:
 \begin{align}
& (\frac{1}{2}) \left(\frac{\gamma}{\beta^{2}}
 \right)
 \vert d( \Psi_{R} g_{\epsilon}^{\beta}) \vert^{2} 
 \leq 
 u g_{\epsilon}^{2 \beta}
  \Psi_{R}^{2}+
 \left(\frac{\gamma}{\beta^{2}}
 \right)
 \vert d \Psi_{R} \vert^{2} g_{\epsilon}^{2\beta} \\ \notag
& + (\frac{1}{2 \beta}) (\triangle \Psi_{R}^{2}) g_{\epsilon}^{2 \beta}
  +
 div 
 \left(
 \Psi_{R}^{2} g_{\epsilon}^{\gamma} dg 
 -(\frac{1}{2 \beta})
 d(\psi_{R}^{2})g_{\epsilon}^{2\beta}
 \right).
 \end{align}
 
 If we let $ g_{R, \epsilon} = \Psi_{R} g_{\epsilon}$, and integrate
\eqref{E:gepsilonweightedsublaplacianineq}, we get
\begin{align}\label{E:reversepoincare}
&(\frac{1}{2})
(\frac{\gamma}{\beta^{2}} )\int_{\Omega} \vert d g_{R, \epsilon} \vert^{2} \, (dx)^{n} \leq \\
&\int_{\Omega} u \vert g_{R, \epsilon} \vert ^{2}  \, (dx)^{n}  + 
 \int_{\Omega} u g_{\epsilon}^{\gamma} \Psi_{R}^{2} \, (dx)^{n} +
C_{\beta} \int_{\Omega} (\vert \triangle ( \Psi_{R}^{2} ) \vert 
+ \vert d \Psi_{R} \vert^{2})
g_{\epsilon}^{2 \beta} \, (dx)^{n}. 
\end{align}
There is no contribution from the divergence term, as
$ ( \Psi_{R}^{2} )g_{\epsilon}^{2} dg + \frac{\gamma}{\beta^{2}} \Psi_{R} [d( \Psi_{R} )] g_{\epsilon}^{2 \beta}$ vanishes on 
$\partial \Omega \cup g^{-1}(\epsilon) \cup \{ s(x) \leq R \}$.
Because $g^{-1}(\epsilon)$ is a smooth submanifold, one can verify the computation at $g^{-1}(\epsilon)$ by a one dimensional argument.
Now from the definition of $\Psi_{R}$ and $\mathcal{S}$
\begin{equation}\label{E:upperestimateoflappsisubRsquared}
\vert \triangle  \Psi_{R}^{2} \vert 
+ \vert d \Psi_{R} \vert^{2}
\leq 
(\frac{1}{R}) \Vert d \Psi^{2} \Vert_{L{\infty}} \vert d^{2} s \vert
+(\frac{1}{R^{2}})+
\left( 
\Vert d^{2} \Psi^{2}
\Vert_{L^{\infty}}
+ \Vert d \Psi \Vert_{L^{\infty}}^{2} \vert ds \vert^{2}
\right)
\leq
(\frac{\hat{K}}{R^{2}}).
\end{equation}
Apply this inequality \eqref{E:upperestimateoflappsisubRsquared} , inequality
\eqref{E:reversepoincare},
and Theorem \ref{T:morreypoincareineq}
to get
\begin{align}\label{E:secondreversepoincareestimate}
& (\frac{1}{2} )( \frac{\gamma}{\beta^{2}} ) 
\int_{\Omega} \vert d g_{R, \epsilon} \vert^{2} \,
(dx)^{n} \leq
\lambda^{-1} \Vert u \Vert_{X^{2}(\Omega)} \int_{\Omega} \vert d g_{R, \epsilon} \vert^{2} \, (dx)^{n} \\ \notag
 & + \epsilon \Vert u \Vert_{L^{2}(\Omega)} \int_{\Omega}
\Vert g_{\epsilon}^{\gamma} 
\Vert_{L^{2}(\Omega)}
+
( C (\beta) ) \hat{K} \bar{K}
\left(
 \int_{s(x) \leq 2R} \vert g_{\epsilon} \vert^{4 \beta} \, (dx)^{n}
 \right)^{\frac{1}{2}}.
\end{align}

Since $ lim_{2R \downarrow 0 } \int_{s(x) \leq 2R } \vert g_{\epsilon}^{4\beta}(dx)^{n} =0$
, if $\Vert u  \Vert_{X^{2}(\Omega	)} < (\frac{\gamma}{2\beta}) \lambda$,
then \\
$g_{0} = \lim_{ \substack	{ R \rightarrow 0 \\ \epsilon \rightarrow 0 }}
g_{R, \epsilon}
=0 $.
\end{proof}
\appendix[{ \bfseries Appendix C Coulomb Gauges}]\label{A:coulombgauges}\\
In order to get further regularity beyond an estimate on $ F_{A} =dA + \frac{1}{2}
[A,A] $, it is necessary to construct a Coulomb gauge, i.e. a local trivialization in which $D_{A} = d+ A$, and $d^{*} A =0$, and to control
the norm of $A$ by a norm of $F_{A}$. Tian and Tao [10] do this in a very weak setting but their proof is very difficult. In our situation we can assume
a bound on $F$ in $L^{p}$ for $p$ large, and their proof simplifies enormously.
We include it here for completeness.
In the following $ B_{y}(\delta) \colon = \{x \colon \vert x-y \vert \leq \delta   \}$, and $B \colon = \{ x \colon \vert x \vert \leq 1  \}$.
\begin{theorem}\label{T:coulombgauge}
Let $\nabla_{\hat{A}} $ be a connection that is smooth on $\Omega - \mathcal{S}$,
, where $\mathcal{S}$ is a closed singular set with finite $n-3$ dimensional
Hausdorff measure. Suppose $F_{\hat{A}} \in L^{p}(\Omega)$ for $p> n$. Then
for each $\epsilon >0$, every point $y \in \Omega$ is the center of a ball 
$B_{y}(\delta) \subset \Omega$, such that 
\begin{equation}\label{E:lpboundonFindeltaball}
\int_{B_{y}(\delta)} \vert F_{\hat{A}} \vert^{p} \, (dx)^{n}
\leq \epsilon^{p} \delta^{-2p-n}
\end{equation}. \\
Note that $\delta$ depends on $\epsilon$.
\end{theorem}

For such an $\epsilon >0$, there exists a local trivialization in which the
connection $\nabla_{A}$ induces a local exterior covariant differential
$D_{A} = d+A$ such that $A$ satisfies $d^{*} A =0$ and
\begin{equation}\label{E:LpboundonAbylpboundonF}
\Vert A \Vert_{L^{p}(B_{y}(\delta))}
\leq 
\delta_{A} C(p,n)
\Vert F_{A} \Vert_{L^{p}(B_{y}(\delta))}=
\delta_{A} C(p,n)
\Vert F_{\hat{A}} \Vert_{L^{p}(B_{y}(\delta))}.
\end{equation}

\begin{proof}
The first statement is clear. Moreover, for interior domains of $\Omega$ there
is a uniform covering of such a domain by such balls. Choose a coordinate system for such a ball, ( $\tilde{x} \colon = \frac{x-y}{\delta}$ ) that transforms this ball into the unit ball.
In this rescaled system
\begin{equation}\label{E:rescaledlpnormofF}
\int_{B} \vert F_{\hat{A}} \vert^{p} \, (d \tilde{x} )^{n} \leq \epsilon^{p}
\end{equation}
 and the conclusion is that there exists a trivialization at the scale of $B$
such that:
\begin{equation}\label{E:rescaledLpboundonAbylpboundonF}
\Vert A \Vert_{L^{p}(B)} \leq C(p,n)
\Vert F_{A} \Vert_{L^{p}(B)}
\leq C(p,n) \epsilon.
\end{equation}
In the rescaled system $0 \, \text{ is not in} \, \mathcal{S}$. Parallel translate the fiber at $0$ along every ray in $B$ until each ray intersects $\mathcal{S}$.
This provides a smooth trivialization of the bundle over $B - \mathcal{S'}$,
where $\mathcal{S'} \colon = \{\lambda y \mid y \in \mathcal{S}, \lambda >0   \} $.
Then, $\mathcal{S'}$ is a closed set of finite $n-2$ dimensional Hausdorff measure. In this gauge (trivialization) $x^{k} A^{k} =0$ (for less cluttered notation we drop the subscript $A$ on $F_{A}$) and
\begin{equation}\label{E:exponentialgaugeestimate}
\frac{\partial}{\partial r}(r A_{j}) = 
A_{j}+ x^{k} F_{k,j} + x^{k} \frac{\partial}{\partial x^{j}} (A_{k} ) =
F_{k,j}.
\end{equation}
Integrating equation \eqref{E:exponentialgaugeestimate} we get
\begin{equation}\label{E:radiallyintegratedexponentialgaugeestimate}
r A_{j} = \int_{0}^{r} \rho F_{\rho,j} \, d \rho
\end{equation}
where $\rho F_{\rho,j} = x^{k} F_{k,j}$. Then
\begin{align}\label{E:firstexponentialgaugeintegralbound}
&
\vert A \vert^{p} r^{p} \leq 
\left(
\int_{0}^{r} \vert F \vert \rho \, d \rho 
\right)^{p}  \leq
\left(
\int_{0}^{r}  \left( \vert F \vert \rho^{\alpha} \right)^{p} \, d \rho 
\right)
\left( \int_{0}^{r}
\left(
\rho^{1-\alpha}
\right)^{\frac{p}{p-1}}
\right)^{p}\\ \notag
&\leq \left(\frac{r}{\beta}
\right)^{p-1} \int_{0}^{r} \vert F \vert^{p} \rho^{\alpha p} \, d \rho.
\end{align}
We set $\alpha = \frac{n-1}{p}$, and compute:
\begin{equation}\label{E:lowerboundonbeta}
\beta 
\left(\frac{(1-\alpha p) }{p-1} +1
\right)	 = \frac{2p- \alpha p -1}{p-1} = \frac{2p-n}{p-1} > 1.
\end{equation}
Integrating $ \vert A \vert^{p} r^{p}$ again in $r$ and also in the spherical angle, gives:
\begin{equation}\label{E:sphericalcordinateformoflpboundofAbylpnormofF}
\int_{0}^{1} \int_{S^{1}} r^{n-1} \vert A \vert^{p} \, d \theta dr
\leq
\int_{0}^{1} r^{p-1} \, dr
\left(
\int_{S^{1}} \int_{0}^{1}  \rho^{n-1} \vert F \vert^{p} \, d \rho d\theta
\right).
\end{equation}
Hence
\begin{equation}\label{E:shortformoflpestimateofAintermsoflpestimateofF}
\Vert A \Vert_{L^{p}(B)} 
\leq
\Vert F \Vert_{L^{p}(B)}
\leq \epsilon.
\end{equation}
Now consider the equation for $g = e^{u}$:
\begin{equation}\label{E:gaugechangeofAequation}
d^{*} ( g^{-1} dg + g^{-1} A g ) =s,
\end{equation}
 which is a smooth map from $\{u, A\}$ (with $u \in
 L^{p}_{1}
 (B) )
\subset C^{0}(B)$ and $A \in L^{p}(B) $) to
$s \in L^{p}_{-1}(B)$.
Since at $u=0$ the linearization 
is $\triangle u + A$, and $\triangle \colon L^{p}_{0}(B) \rightarrow
L^{p}_{-1}(B)$ is invertible, equation \eqref{E:gaugechangeofAequation} is solvable for small $A$ for $u$ near $0$ in $L^{p}_{1,0}(B)$.
Now, in the new gauge, $\tilde{A} =g^{-1} dg + g^{-1} A g $, and
$F_{\tilde{A}} = g^{-1} F_{A} g$. We have:
\begin{gather}
	d^{*} \tilde{A} = 0 \\
	d \tilde{A} + (\frac{1}{2}) [\tilde{A}, \tilde{A} ] = F_{\tilde{A}}.
\end{gather}
and,
\begin{align}\label{E:lpnormofAtildeestimatedbylpnormofFtilde}
&\Vert \tilde{A} \Vert_{L^{p}(B)} \leq 
\Vert dg \Vert_{L^{p}(B)} + \Vert A \Vert_{L^{p}(B)} \\ \notag
&\leq  (C_{p} +1) \Vert A \Vert_{L^{p}(B)} 
\leq  (C_{p} +1) \Vert F_{A} \Vert_{L^{p}(B)}
\leq  (C_{p} +1) \Vert \tilde{F}_{A} \Vert_{L^{p}(B)} 
\end{align}
 Here $C_{p}$ is roughly the norm of the inversion of
$\triangle \colon L^{p}_{1,0}(B) \rightarrow L^{p}_{-1}(B)$.
\end{proof}
This construction gives a gauge transformation between smooth connections on
$B -\mathcal{S}$.
Because the singular set $\mathcal{S}$ is of Hausdorff codimension
at least two, the gauge transformation, which is a map to a compact group,
must extend over $B- \mathcal{S}'$.
By the same argument, the singularities can only be removed in one way.
\begin{corollary}\label{C:interiorregularityestimate}
On $B(\frac{1}{2})$, we have $\tilde{A} \in L^{p}_{1}(B(\frac{1}{2}) )$.
\end{corollary}
\begin{proof}
Since
\begin{equation}
\Vert \tilde{A} \Vert_{L^{p}(B)} \leq 
 (C_{p} +1) \Vert \tilde{F}_{A} \Vert_{L^{p}(B)},
\end{equation}
and since $\tilde{A}$ solves the elliptic system:
\begin{gather}
d^{*} \tilde{A} =0 \\
d \tilde{A} + \frac{1}{2}[ \tilde{A}, \tilde{A} ] = \tilde{F}, 
\end{gather}
 standard interior elliptic regularity gives an estimate of
$\Vert \tilde{A} \Vert_{L^{p}_{1}(B(\frac{1}{2}) }$ in terms
of 
$\Vert F_{\tilde{A} }\Vert_{L^{p}(B )}$.
\end{proof}
Note that we do not set the radial $x^{i}A_{i}$ equal to zero on the boundary. We only use an interior regularity estimate.

 \end{document}